\input ./style/arxiv-general.cfg
\documentclass[MSNbibl,number,citesort,dvips]{arxbj}
\makeatletter
   \@ifpackageloaded{graphicx}{}{\usepackage{graphicx}}
\makeatother
\usepackage{mathbh}

\aid{0}
\volume{21}
\issue{3}
\pubyear{2015}
\firstpage{1412}
\lastpage{1434}
\doi{10.3150/14-BEJ608} 

\makeatletter
\newcommand{\rrvert}{\vert}
\newcommand{\llvert}{\vert}
\renewcommand{\emptyset}{\varnothing}
\newproclaim{definition}{Definition}[section]
\newremark{remarknorm}[definition]{Remark}
\newtheorem{theorem}[definition]{Theorem}
\newremark{examplenorm}[definition]{Example}
\newtheorem{corollary}[definition]{Corollary}
\newcommand{\N}{\mathbb{N}}
\newcommand{\Z}{\mathbb{Z}}
\newcommand{\Q}{\mathbb{Q}}
\newcommand{\R}{\mathbb{R}}
\newcommand{\I}{\mathbb{I}}
\newcommand{\pr}{\mathbb{P}}
\newcommand{\prq}{\mathbb{Q}}
\newcommand{\ex}{\mathbb{E}}
\newcommand{\eins}{\mathbh{1}}
\newcommand{\T}{\mathbf{T}}
\makeatother

\begin{document}
\begin{frontmatter}

\title{Qualitative robustness of statistical functionals under strong mixing}
\runtitle{Qualitative robustness of statistical functionals}

\begin{aug}
\author[1]{\inits{H.}\fnms{Henryk}~\snm{Z\"ahle}\corref{}\ead[label=e1]{zaehle@math.uni-sb.de}}
\address[1]{Department of Mathematics, Saarland University, Postfach
151150, D-66041 Saarbr\"{u}cken, Germany.\\ \printead{e1}}
\end{aug}

\received{\smonth{2} \syear{2012}}
\revised{\smonth{1} \syear{2013}}

%
\begin{abstract}
A new concept of (asymptotic) qualitative robustness for plug-in
estimators based on identically distributed possibly \emph{dependent}
observations is introduced, and it is shown that Hampel's theorem for
general metrics $d$ still holds. Since Hampel's theorem assumes the UGC
property w.r.t. $d$, that is, convergence in probability of the
empirical probability measure to the true marginal distribution w.r.t.
$d$ uniformly in the class of all admissible laws on the sample path
space, this property is shown for a large class of strongly mixing laws
for three different metrics $d$. For real-valued observations, the UGC
property is established for both the Kolomogorov $\phi$-metric and the
L\'{e}vy $\psi$-metric, and for observations in a general locally
compact and second countable Hausdorff space the UGC property is
established for a certain metric generating the $\psi$-weak topology.
The key is a new uniform weak LLN for strongly mixing random variables.
The latter is of independent interest and relies on Rio's maximal inequality.
\end{abstract}

%

\end{frontmatter}

\section{Introduction}\label{Introduction}

Let $\mathcal{M}$ be a class of probability measures on some
measurable space $E$, and $T$ be a mapping (statistical functional)
from $\mathcal{M}$ into a measurable space $\T$. Let $(X_i)_{i\in\N
}$ be a sequence of $E$-valued random elements (observations) being
identically distributed according to $\mu\in\mathcal{M}$. If
$\widehat m_n=\frac{1}{n}\sum_{i=1}^n\delta_{X_i}$ denotes the
empirical distribution of $X_1,\ldots,X_n$, then $T(\widehat m_n)$ can
provide a reasonable estimator for $T(\mu)$. Informally, the sequence
$(T(\widehat m_n))$ is qualitatively robust when for large $n$ a small
change in $\mu$ results only in a small change of the law of the
estimator $T(\widehat m_n)$. More precisely, given a subset $\mathcal
{P}_1\subset\mathcal{M}$, the sequence of estimators $(T(\widehat
m_n))$ is said to be qualitatively $\mathcal{P}_1$-robust at $\mu\in
\mathcal{P}_1$ if for every $\varepsilon>0$ there are some $\delta
>0$ and $n_0\in\N$ such that
%
\begin{equation}
\label{def qual rob in intro} \nu\in\mathcal{P}_1,\qquad d(\mu,\nu)\le\delta\quad
\Longrightarrow\quad d'\bigl(\operatorname{law}\bigl\{T(\widehat
m_n)|\mu\bigr\},\operatorname{law}\bigl\{T(\widehat m_n)|\nu
\bigr\}\bigr)\le\varepsilon\qquad\forall n\ge n_0,
\end{equation}
where $d$ and $d'$ are metrics on $\mathcal{P}_1$ and on the class of
all probability measures on $\T$, respectively. The basic concept of
qualitative robustness was initiated by Hampel \cite
{Hampel1968,Hampel1971}, but the above version of qualitative
robustness is due to Huber \cite{Huber1981}. Huber's version is also
called asymptotic robustness (cf. \cite{PapantoniKazakosGray1979}) and
differs from the original definition in \cite{Hampel1968,Hampel1971}
in that the right-hand side in (\ref{def qual rob in intro}) is not
required to hold for all $n\in\N$ but only for all $n\ge n_0$ for
some $n_0=n_0(\varepsilon)$. For background see also \cite
{Cuevas1988,Hampeletal1986,Huber1981,HuberRonchetti2009,Kraetschmeretal2012a,Mizera2010,Rieder1994}
and references therein.

The definition of qualitative robustness stated above was introduced by
Hampel and Huber to capture the case of independent observations. To
capture also the case of dependent observations, various authors
departed from this definition and considered, instead of a metric $d$
on a class of probability measures on $E$, a metric $d_\N$ on a class
of probability measures on the infinite product space $E^{\N}$ or a
sequence of pseudo-metrics $(d_n)$ on classes of probability measures
on $E^n$, $n\in\N$; cf. \cite
{Boenteatal1987,Bustos1981,Cox1981,Hampel1971,PapantoniKazakosGray1979}.
However, in the usual situation where one is interested in the
estimation of an aspect $T(\mu)$ of the marginal distribution $\mu$
based on the observations $X_1,\ldots,X_n$, and where the contaminated
observations are still identically distributed (according to some $\nu
$ close to $\mu$), it might be also reasonable to retain the original
definition. After all, for increasing sample size $n$ the impact of the
data dependence on the estimation often declines. So one might hope
that if the dependence structures induced by the ``admissible''
probability measures on $E^{\N}$ (which play the role of the laws of the
sequences of identically distributed observations) are subject to a
common constraint, then the implication (\ref{def qual rob in intro})
still holds for the class $\mathcal{P}_1$ of the marginal
distributions. In other words, under a common constraint for the
dependence structures, it might be sufficient to ensure a small
distance between the marginal distributions $\mu$ and $\nu$ in order
to obtain a small distance between the distributions of the plug-in
estimators, based on large $n$, under two ``admissible'' laws on $E^{\N}
$ with marginal distributions $\mu$ and $\nu$, respectively.

In this article, we will demonstrate that the latter is in fact true
under fairly weak constraints for the dependence structures. In Section~\ref{section def m robustness}, we adapt Huber's definition of
qualitative robustness to the case of dependent observations and
establish the analogue of Hampel's theorem. Since the latter assumes
the UGC property, that is, convergence in probability of the empirical
probability measure to the true marginal distribution uniformly in the
class of all ``admissible'' laws on $E^{\N}$ (cf. Definition~\ref{def
UGC property} below), this property will be established for a large
class of strongly mixing laws on $E^{\N}$ for three different metrics.
In Section~\ref{UGC weighted kolmogorov}, we consider real-valued
observations ($E=\R$) and verify the UGC property for both the
Kolomogorov $\phi$-metric and the L\'{e}vy metric. In Section~\ref{UGC psi weak topology}, we assume observations in a general locally
compact and second countable Hausdorff space $E$ and verify the UGC
property for a certain metric generating the $\psi$-weak topology. For
both examples the key is a new uniform weak LLN for strongly mixing
random variables which is given in the Appendix~\ref{appendix rio} and
which relies on Rio's maximal inequality.

It should be stressed that for the considerations of this article it is
essential that the definition of qualitative robustness (Definition~\ref{def quali rob} below) is in line with Huber's version of
qualitative robustness \cite{Huber1981}, that is, with asymptotic
robustness. Indeed, from Example~1.18 in \cite{Bradley2007} it is
easily seen that for \emph{fixed} $n$, weak dependence can change the
distribution of an estimator even if the marginal distributions of the
observed data are the same. In this respect, the intension of this
article differs from the objective of the existing literature on
qualitative robustness for dependent observations \cite
{Boenteatal1987,Bustos1981,Cox1981,Hampel1971,PapantoniKazakosGray1979}
where the estimators are demanded to be ``stable'' not only for large
but also for small samples (and are allowed to be more general than
plug-in estimators). The latter notion of robustness requires a more
sophisticated definition compared to Definition~\ref{def quali rob}
below. See, for instance, \cite{Boenteatal1987} for an informative
discussion on a proper choice for the definition of qualitative
robustness in this context.


\section{Qualitative robustness and a Hampel theorem}\label{section
def m robustness}

Let $(E,\mathcal{E})$ be a measurable space, $\Omega:=E^{\N}$,
$\mathcal{F}:=\mathcal{E}^{{\N}}$, $X_i$ be the $i$th coordinate
projection on $\Omega$, and $\pr_i:=\pr\circ X_i^{-1}$ be the $i$th
marginal distribution of a probability measure $\pr$ on $(\Omega
,\mathcal{F})$. Let $\mathcal{P}$ be a class of probability measures
on $(\Omega,\mathcal{F})$ such that $\pr_1=\pr_2=\cdots$ for every
$\pr\in\mathcal{P}$. Let $\mathcal{P}_1:=\{\pr_1\dvtx \pr\in\mathcal
{P}\}$ be the corresponding class of all marginal distributions, and
$\mathcal{M}$ be any subset of the class $\mathcal{M}_1(E)$ of all
probability measures on $(E,\mathcal{E})$ such that $\mathcal
{P}_1\subset\mathcal{M}$. Let $(\T,\mathcal{T})$ be a measurable
space, and $T\dvtx \mathcal{M}\to\T$ be a mapping (statistical
functional). For every $n\in\N$, we assume that the mapping
%
\begin{equation}
\label{def estimator in our setting} \widehat T_n(x) = \widehat T_n
\bigl(x^{(n)}\bigr) := T(\widehat m_{x^{(n)}}),\qquad
x=(x_1,x_2,\ldots)\in\Omega
\end{equation}
is $(\mathcal{E}^{\N},\mathcal{T})$-measurable, where $\widehat
m_{x^{(n)}}:=\frac{1}{n}\sum_{i=1}^n\delta_{x_i}$ denotes the
empirical probability measures associated with $x^{(n)}:=(x_1,\ldots
,x_n)$. For (\ref{def estimator in our setting}) to be well defined,
we assume that the set of all such empirical probability measures is
contained in $\mathcal{M}$. Notice that $\widehat T_n$ provides an
estimator for $T(\mu)$. We let $d'$ be some metric on the set
$\mathcal{M}_1(\T)$ of all probability measures on $(\T,\mathcal
{T})$, and $d$ be some metric on $\mathcal{P}_1$.

\begin{definition}[(Qualitative robustness)]\label{def quali rob}
Let\vspace*{2pt} us take the notation from above, and let $\pr\in\mathcal{P}$.
Then the sequence $(\widehat T_n)$ of estimators is said to be \emph
{qualitatively $\mathcal{P}$-marginally robust} at $\pr$ w.r.t.
$(d,d')$ if for every $\varepsilon>0$ there are some $\delta>0$ and
$n_0\in\N$ such that
\[
\prq\in\mathcal{P},\qquad d(\pr_1,\prq_1)\le\delta\quad
\Longrightarrow\quad d'\bigl(\pr\circ\widehat T_n^{-1},
\prq\circ\widehat T_n^{-1}\bigr)\le \varepsilon\qquad
\forall n\ge n_0.
\]
\end{definition}

Sometimes also the functional $T$ itself will be called qualitatively
$\mathcal{P}$-marginally robust at $\pr$ if the corresponding
sequence of plug-in estimators $(\widehat T_n)$ is.

\begin{remarknorm}\label{unabhaeniger fall}
If every $\pr\in\mathcal{P}$ is an infinite product measure, that
is, $\pr=\pr_1^{\N}$, then Definition~\ref{def quali rob} coincides
with the classical definition of qualitative robustness for independent
observations; cf. \cite
{Huber1981,HuberRonchetti2009,Kraetschmeretal2012a}.
\end{remarknorm}

In applications, the validity of qualitative robustness in the sense of
Definition~\ref{def quali rob} is typically hard to check
``directly''. So it is natural to ask for transparent sufficient
conditions. In the framework of Remark~\ref{unabhaeniger fall}, the
celebrated Hampel theorem provides a sufficient condition for
qualitative robustness when $\mathcal{M}=\mathcal{P}_1=\mathcal
{M}_1(E)$, $E$ is Polish, and $d$ and $d'$ are the Prohorov metrics;
cf. \cite
{Cuevas1988,Hampel1971,Huber1981,HuberRonchetti2009,Mizera2010}. In
\cite{Kraetschmeretal2012a}, this result was extended to more general
metric spaces $(\mathcal{P}_1,d)$ but still in the framework of Remark~\ref{unabhaeniger fall}. In Theorem~\ref{hampel-huber generalized}
below, we will formulate a version of Hampel's criterion which can also
be applied to our general setting. As usual, we choose $d'$ as the
Prohorov metric. To this end, we assume that $\T$ is equipped with a
complete and separable metric $d_\T$ and that $\mathcal{T}$ is the
corresponding Borel $\sigma$-field. Recall that the Prohorov metric is
given by
\[
d_{\mathrm{{Proh}}}'(\mu,\nu) := \inf\bigl\{\varepsilon>0 \dvtx
\mu[A]\le \nu\bigl[A^\varepsilon\bigr]+\varepsilon\mbox{ for all }A\in
\mathcal{T}\bigr\},
\]
where $A^\varepsilon:=\{t\in T\dvtx  \inf_{a\in A}d_\T(t,a)\le
\varepsilon\}$ is the $\varepsilon$-hull of $A$. Moreover, we set
$\widehat m_n(x):=\widehat m_{x^{(n)}}$ for all $x\in\Omega$, and
assume that $x\mapsto d(\widehat m_n(x),\pr_1)$ is $(\mathcal
{F},\mathcal{B}(\R_+))$-measurable. In the following definition,
which is a generalization of Definition~2.3 in \cite
{Kraetschmeretal2012a}, the acronym UGC stands for ``uniform (weak)
Glivenko--Cantelli''.

\begin{definition}[(UGC property)]\label{def UGC property}
We say that $\mathcal{P}$ \emph{admits the UGC property w.r.t. $d$}
if for every $\delta>0$
%
\begin{equation}
\label{hampel-huber generalized - proof - 2} \lim_{n\to\infty} \sup_{\pr\in\mathcal{P}} \pr
\bigl[ d(\widehat m_n,\pr_1)\ge\delta \bigr] = 0.
\end{equation}
\end{definition}

\begin{theorem}[(Hampel-type theorem)]\label{hampel-huber generalized}
Assume that $\mathcal{P}$ admits the UGC property w.r.t. $d$, and let
$\pr\in\mathcal{P}$. Then, if the mapping $T$ is continuous at $\pr
_1$ w.r.t. $(d,d_\T)$, the sequence $(\widehat T_n)$ is qualitatively
$\mathcal{P}$-marginally robust at $\pr$ w.r.t. $(d,d_{\mathrm{{Proh}}}')$.
\end{theorem}

The proof of Theorem~\ref{hampel-huber generalized} can be found in
Section~\ref{proof of hampel-huber generalized}.

\begin{remarknorm}
Let $\mathcal{M}_{1,\mathrm{emp}}$ denote the space of all empirical
probability measures $\widehat m_{x^{(n)}}$ with $x\in E^{\N}$ and $n\in
\N$, and recall that we assumed $\mathcal{M}_{1,\mathrm{emp}}\subset
\mathcal{M}$. It is clear from the proof in Section~\ref{proof of
hampel-huber generalized} that in Theorem~\ref{hampel-huber
generalized} it suffices to require that $T$ is $\mathcal
{M}_{1,\mathrm{emp}}$-continuous at $\pr_1$ w.r.t. $(d,d_\T)$,
meaning that for every $\varepsilon>0$ there is some $\delta>0$ such
that for all $\nu\in\mathcal{M}_{1,\mathrm{emp}}$ with $d(\pr
_1,\nu)\le\delta$ we have that $d_\T(T(\pr_1),T(\nu))\le
\varepsilon$.
\end{remarknorm}

\begin{remarknorm}\label{hampel-huber conversed}
Adapting the proof of Theorem~2.6 in \cite{Kraetschmeretal2012a}, one
also obtains a sort of converse of Hampel's theorem. More precisely,
fix $\pr\in\mathcal{P}$ and assume that $(\widehat T_n)$ is weakly
consistent w.r.t. $d_\T$ (i.e., $\widehat T_n$ converges in $\Q
$-probability to $T(\Q_1)$ w.r.t. $d_\T$) at every $\Q\in\mathcal
{P}$ for which $\Q_1$ lies in some neighborhood of $\pr_1$ w.r.t.
$d$. Then qualitative $\mathcal{P}$-marginal robustness of the
sequence $(\widehat T_n)$ at $\pr$ w.r.t. $(d,d_{\mathrm{{Proh}}}')$
implies that the mapping $T$ is $\mathcal{P}_1$-continuous at $\pr_1$
w.r.t. $(d,d_\T)$. The latter means that for every $\varepsilon>0$
there is some $\delta>0$ such that $d_\T(T(\pr_1),T(\nu))\le
\varepsilon$ for every $\nu\in\mathcal{P}_1$ with $d(\pr_1,\nu
)\le\delta$.
\end{remarknorm}

In Section~\ref{section examples}, we will give two examples for
classes of probability measures on $(\Omega,\mathcal{F})$ admitting
the UGC property. To motivate the Hampel-type Theorem~\ref
{hampel-huber generalized},\vadjust{\goodbreak} we will also discuss continuity of
particular statistical functionals w.r.t. the involved metrics.

The classical choice for $d$ in the framework of Definition~\ref{def
quali rob} is any metric generating the weak topology. The most
prominent examples are the Prohorov metric and the L\'{e}vy metric used
in \cite
{Cuevas1988,Hampel1971,Huber1981,HuberRonchetti2009,Mizera2010} and
many further references. Another example is the bounded Lipschitz
metric used, for instance, in \cite
{Christmannetal2011,CuevasRomo1993,Huber1981,HuberRonchetti2009}.
However, for some purposes it is somewhat restrictive to use
exclusively a metric generating the weak topology. The use of such a
metric creates a sharp division of the class of statistical functionals
$T$ into those for which $(\widehat T_n)$ is ``robust'' and those for
which $(\widehat T_n)$ is ``not robust''. Indeed, Hampel's theorem says
that $(\widehat T_n)$ is ``robust'' if and only if $T$ is continuous
w.r.t. the weak topology. But the distributions of the plug-in
estimators of two statistical functionals being not continuous w.r.t.
the weak topology may react quite different to changes in the
underlying (marginal) distribution, just as these plug-in estimators
may have quite different influence functions. For this reason, it was
proposed in \cite{Kraetschmeretal2012a} to investigate $(\widehat
T_n)$ for qualitative robustness w.r.t. more general metrics, where it
is clear that qualitative robustness w.r.t. a metric $d_1$ is a
stronger condition than qualitative robustness w.r.t. a metric $d_2\le
d_1$. In particular, a statistical functional $T_1$ can be considered
to have a higher ``degree of robustness'' than another statistical
functional $T_2$ when $T_1$ is qualitatively robust for any choice of
$d$ for which $T_2$ is qualitatively robust. In this way, it gets
possible to differentiate plug-in estimators w.r.t. qualitative
robustness within the class of statistical functionals that are not
weakly continuous. Sensible classes of metrics that can be studied in
this context are, for instance, $\{d_{(\phi)}\}_\phi$ and $\{d_{\psi
,\mathrm{{vag}}}\}_\psi$ to be introduced in (\ref{dphi Definition
Eq}) and (\ref{def psi weak metric}), respectively, where $\phi$ and
$\psi$ (or rather their increases) can be seen as gauges for the
strictness of the metric and thus for the ``degree of robustness''. For
details see \cite{Kraetschmeretal2012a,Kraetschmeretal2012b}. The
latter reference provides in particular a rigorous quantification of
the ``degree of robustness'' for convex risk functionals.

The preceding discussion counters somewhat the conventional point of
view that the sequence $(\widehat T_n)$ of plug-in estimators can be
considered to be qualitatively robust exclusively when $T$ is
continuous w.r.t. the weak topology. But even if one insists on the use
of the weak topology, the considerations below provide new results. For
instance, Corollary~\ref{glivenko cantelli - corollary} and Theorem~\ref{hampel-huber generalized} together yield a nontrivial
generalization of the classical Hampel theorem in the form of \cite
{Huber1981}, Theorem~2.21. In this theorem the underlying metric is the
L\'{e}vy metric which generates the weak topology.

In the sequel we will repeatedly work with left- and right-continuous
inverses. Recall that the left-continuous inverse of any nondecreasing
function $H\dvtx \R\to\R$ is defined by $H^\leftarrow(t):=\inf\{y\in\R
\dvtx H(y)\ge t\}$ with the convention $\inf\emptyset=\infty$. The
right-continuous inverse $H^\rightarrow$ of any nonincreasing function
$H\dvtx \R_+\to\R_+$ is defined by $H^\rightarrow(t):=\sup\{y\in\R
_+\dvtx H(y)>t\}$ with the convention $\sup\emptyset:=0$. We will also
repeatedly use the notion of strong mixing which is recalled in the
Appendix~\ref{appendix rio}. The $n$th strong mixing coefficient of
the coordinate process $(X_i)$ under the law $\pr$ will be denoted by
$\alpha_\pr(n)$.


\section{Examples}\label{section examples}


\subsection{Strong mixing, and Kolmogorov \texorpdfstring{$\phi$}{phi}-metric or L\'{e}vy
metric}\label{UGC weighted kolmogorov}

In this section, we will see that in the case $E=\R$ a large class of
probability measures on $(\Omega,\mathcal{F})$ admits the UGC
property w.r.t. (a weighted version of) the Kolmogorov metric. As a
corollary we will also obtain the UGC property w.r.t. the L\'{e}vy
metric. Let $\phi$ be a \emph{u-shaped function}, that is, a
continuous function $\phi\dvtx \R\to[1,\infty)$ that is nonincreasing on
$(-\infty,0)$ and nondecreasing on $(0,\infty)$. Then
%
\begin{equation}
\label{dphi Definition Eq} d_{(\phi)}(\mu,\nu) := \sup_{y\in\R} \bigl
\llvert F_{\mu}(y)-F_\nu(y)\bigr\rrvert \phi(y)
\end{equation}
defines a metric on the set $\mathcal{M}_1^{(\phi)}(\R)$ of all
probability measures $\mu$ on $(\R,\mathcal{B}(\R))$ for which
$d_{(\phi)}(\mu,\delta_0)<\infty$. We will refer to $d_{(\phi)}$
as \emph{Kolmogorov $\phi$-metric}. Notice that $\mu\in\mathcal
{M}_1^{(\phi)}(\R)$ if $\int\phi \,\mathrm{d}\mu<\infty$, and that $d_{(\phi
)}$ is just the classical Kolmogorov metric for $\phi:=\eins$.

In \cite{BeutnerZaehle2013,Zaehle2011} it is demonstrated that many L-
and V-functionals $T$ as well as many coherent risk functionals $T$ are
continuous w.r.t. $d_{(\phi)}$ with $\phi$ depending on the
particular functional $T$; see also Example~\ref{remark on cond b -
role of phi} for some simple examples. So, in view of the Hampel-type
Theorem~\ref{hampel-huber generalized}, for a given class $\mathcal
{P}$ of probability measures on $(\Omega,\mathcal{F})$, the
corresponding sequence $(\widehat T_n)$ of plug-in estimators is
qualitatively $\mathcal{P}$-marginally robust w.r.t. $(d_{(\phi
)},d_{\mathrm{{Proh}}}')$ at any $\pr\in\mathcal{P}$ if $\mathcal
{P}$ admits the UGC property w.r.t. $d_{(\phi)}$ in the sense of
Definition~\ref{def UGC property}. The following Theorem~\ref
{glivenko cantelli} shows that the latter is true if under every $\pr
\in\mathcal{P}$ the coordinates of the coordinate process $(X_i)$ are
identically distributed and strongly mixing with uniformly (in $\pr\in
\mathcal{P}$) decaying mixing coefficients $(\alpha_\pr(n))$ and if
the class of marginal distributions $\mathcal{P}_1$ is uniformly $\phi
$-integrating. It is remarkable that the common rate of decay of the
mixing coefficients may be arbitrarily slow.

\begin{theorem}[(UGC property w.r.t. $\boldsymbol{d}_{\boldsymbol{(\phi)}}$)]\label
{glivenko cantelli}
Let $\phi$ be a u-shaped function, and $\mathcal{P}$ be a class of
probability measures on $(\Omega,\mathcal{F})$ such that
\begin{enumerate}[(b)]
\item[(a)] $\pr_1=\pr_2=\cdots$ for all $\pr\in\mathcal{P}$,
\item[(b)] $\lim_{K\to\infty}\sup_{\pr\in\mathcal{P}}\int\phi
(y)\eins_{\phi(y)\ge K} \pr_1(\mathrm{d}y)=0$,
%
\item[(c)] $\lim_{n\to\infty}\sup_{\pr\in\mathcal{P}}\alpha
_\pr(n)=0$.
\end{enumerate}
Then, for every $\delta>0$,
%
\begin{equation}
\label{glivenko cantelli - eq} \lim_{n\to\infty} \sup_{\pr\in\mathcal{P}} \pr
\bigl[ d_{(\phi
)}(\widehat m_{n},\pr_1)\ge\delta
\bigr] = 0.
\end{equation}
\end{theorem}

The proof of Theorem~\ref{glivenko cantelli} can be found in Section~\ref{proof of glivenko cantelli}. Notice that (c) implies in
particular that the coordinate process $(X_i)$ is strongly mixing under
every $\pr\in\mathcal{P}$.

\begin{remarknorm}\label{remark on cond b}
(i) Condition (b) is always fulfilled if $\phi$ is bounded, in
particular if $d_{(\phi)}$ is the classical Kolmogorov metric ($\phi
=\eins$).

(ii) Condition (b) is also fulfilled if one can find some function
$w\dvtx \R\to[0,\infty)$ such that $\lim_{|x|\to\infty}w(x)/\phi
(x)=\infty$ and $\sup_{\pr\in\mathcal{P}}\int w(y) \pr
_1(\mathrm{d}y)<\infty$; cf. \cite{Kraetschmeretal2012b}, Lemma~4.1.
\end{remarknorm}

\begin{remarknorm}\label{remark on cond b - two}
Since the L\'{e}vy metric $d_{\mbox{\scriptsize{{L\'{e}vy}}}}$ (defined in (\ref{def levy metric}) below) generates the weak topology (cf. \cite{Huber1981}, page
25) and the Kolmogorov metric $d_{(\eins)}$ dominates the
L\'{e}vy metric, that is, $d_{\mbox{\scriptsize{{L\'{e}vy}}}}\le d_{(\eins)}$
(cf. \cite{Huber1981}, page 34), we obtain that every functional $T$
which is weakly continuous at some $\mu\in\mathcal{M}_1(\R)$ is
also continuous w.r.t. $d_{(\eins)}$ at $\mu$. Further notice that
qualitative robustness w.r.t. $d_{\mbox{\scriptsize{{L\'{e}vy}}}}$ implies
qualitative robustness w.r.t. $d_{(\eins)}$. On the other hand, the
UGC property w.r.t. $d_{(\eins)}$ implies the UGC property w.r.t.
$d_{\mbox{\scriptsize{{L\'{e}vy}}}}$.
\end{remarknorm}

\begin{examplenorm}\label{remark on cond b - role of phi}
(i) It is well known that for fixed $\alpha\in(0,1)$ the lower
$\alpha$-quantile functional $T_\alpha(\mu):=F_\mu^\leftarrow
(\alpha)$ is continuous w.r.t. the weak topology at $\mu\in\mathcal
{M}_1(\R)$ when $F^{\leftarrow}_{\mu}$ is continuous at $\alpha$;
see, for instance, \cite{vanderVaart1998}, Lemma~21.2. According to
the first part of Remark~\ref{remark on cond b - two}, $T_\alpha$ is
in particular continuous w.r.t. the Kolmogorov metric $d_{(\eins)}$ at
this $\mu$.

(ii) The mean functional $T^{(1)}(\mu):=\int y \mu(\mathrm{d}y)$ is continuous
on $\mathcal{M}^{(\phi)}$ w.r.t. $d_{(\phi)}$ for any $\phi$
satisfying $\int1/\phi(y) \,\mathrm{d}y<\infty$. This follows from the
inequality $|T^{(1)}(\mu)-T^{(1)}(\nu)|\le\int|F_\mu(y)-F_\nu(y)|
\,\mathrm{d}y$ which holds for all $\mu,\nu\in\mathcal{M}^{(\phi)}$. For
instance, we may choose $\phi(y)=(1+|y|)^{1+\varepsilon}$ for
arbitrarily small $\varepsilon>0$. In this case, condition (b) in
Theorem~\ref{glivenko cantelli} holds when $\sup_{\pr\in\mathcal
{P}}\int|y|^{1+\varepsilon'} \pr_1(\mathrm{d}y)<\infty$ for some
$\varepsilon'>\varepsilon$; cf. Remark~\ref{remark on cond b}(ii).

(iii) The second moment functional $T^{(2)}(\mu):=\int y^2 \mu(\mathrm{d}y)$,
and thus the variance functional, is continuous on $\mathcal{M}^{(\phi
)}$ w.r.t. $d_{(\phi)}$ for any $\phi$ satisfying $\int|y|/\phi(y)
\,\mathrm{d}y<\infty$. This follows from the inequality $|T^{(2)}(\mu
)-T^{(2)}(\nu)| \le2\int|F_\mu(y)-F_\nu(y)| |y| \,\mathrm{d}y$ which holds
for all $\mu,\nu\in\mathcal{M}^{(\phi)}$. For instance, we may
choose $\phi(y)=(1+|y|)^{2+\varepsilon}$ for arbitrarily small
$\varepsilon>0$. In this case, condition (b) in Theorem~\ref{glivenko
cantelli} holds when $\sup_{\pr\in\mathcal{P}}\int
|y|^{2+\varepsilon'} \pr_1(\mathrm{d}y)<\infty$ for some $\varepsilon
'>\varepsilon$; cf. Remark~\ref{remark on cond b}(ii).
\end{examplenorm}

\begin{examplenorm}[(Linear processes)]\label{linear process}
Let $(Z_s)_{s\in\Z}$ be a sequence of i.i.d. random variables on any
probability space, assume that $|Z_1|$ has a finite expectation denoted
by $L$, and assume that the distribution of $Z_1$ admits a Lebesgue
density $f$ for which $\int|f(y+h)-f(y)| \,\mathrm{d}y<M|h|$ for all $h\in\R$
and some constant $M>0$. Define the linear process $X_t^{a} := \sum_{s=0}^\infty a_{s}Z_{t-s}$, $t\in\N$, for any real sequence
$a=(a_s)_{s\in\N_0}$, and let $A$ be the class of all real sequences
$a$ for which $a_0=1$ and $\sum_{s=0}^\infty a_{s}Z_{t-s}$ is almost
surely absolutely convergent for every $t\in\N$. Results in \cite
{PhamTran1985} imply that, if $a\in A$ satisfies $\sum_{s=0}^\infty
a_s z^s\neq0$ for all $z$ with $|z|\le1$, and $\sum_{u=1}^\infty
\sum_{s=u}^\infty|a_s|<\infty$, then $(X_t^a)$ is strongly mixing
with mixing coefficients $(\alpha_a(n))$ satisfying $\alpha_a(n)\le
(2ML\sum_{s=0}^\infty|b_s(a)|)\sum_{u=n}^\infty\sum_{s=u}^\infty
|a_s|$, where $b_s(a)$ is the coefficient of $z^s$ in the power series
expansion of $z\mapsto1/\sum_{s=0}^\infty a_sz^s$, and $\sum_{s=0}^\infty|b_s(a)|<\infty$; cf. the Appendix~\ref{Strong mixing
of linear processes}. If we denote by $\pr^{a}$ the law of $(X_t^{a})$
on $\R^{\N}$, then, of course, the coordinate process $(X_t)$ on $\R
^{\N}$ is also strongly mixing under $\pr^a$ with mixing coefficients
$(\alpha_{\pr^a}(n))$ satisfying $\alpha_{\pr^a}(n)\le(2ML\sum_{s=0}^\infty|b_s(a)|)\sum_{u=n}^\infty\sum_{s=u}^\infty|a_s|$.

Now, let $A'$ be any subset of $A$ such that
\begin{enumerate}[(iii)]
\item[(i)] $\sum_{s=0}^\infty a_s z^s\neq0$ for all $z$ with
$|z|\le1$, for every $a\in A'$,
\item[(ii)] $\lim_{n\to\infty}\sup_{a\in A'}\sum_{u=n}^\infty
\sum_{s=u}^\infty|a_s|=0$,
\item[(iii)] $\sup_{a\in A'}\sum_{s=0}^\infty|b_s(a)|<\infty$.
\end{enumerate}
Then we obtain from the statement above that the class $\mathcal
{P}=\mathcal{P}':=\{\pr^a\dvtx a\in A'\}$ satisfies condition (c) in
Theorem~\ref{glivenko cantelli}. Moreover, condition (a) in Theorem~\ref{glivenko cantelli} is fulfilled for $\mathcal{P}=\mathcal{P}'$
anyway, and condition (b) in Theorem~\ref{glivenko cantelli} is always
fulfilled when $\phi=\eins$. Therefore, Theorem~\ref{glivenko
cantelli} shows that $\mathcal{P}'$ admits the UGC property for the
Kolmogorov metric $d_{(\eins)}$. In particular, the Hampel-type
Theorem~\ref{hampel-huber generalized} implies that any $d_{(\eins
)}$-continuous functional $T$ is qualitatively $\mathcal
{P}'$-marginally robust at any $\pr^a\in\mathcal{P}'$ w.r.t.
$(d_{(\eins)},d_{\mathrm{{Proh}}}')$. That is, for every $\varepsilon
>0$ there are some $\delta>0$ and $n_0\in\N$ such that
\[
\pr^{a'}\in\mathcal{P}',\qquad d_{(\eins)}\bigl(
\pr_1^a,\pr_1^{a'}\bigr)\le \delta
\quad\Longrightarrow\quad d_{\mathrm{{Proh}}}'\bigl(\pr^a
\circ\widehat T_n^{-1}, \pr^{a'}\circ\widehat
T_n^{-1}\bigr)\le\varepsilon \qquad\forall n\ge
n_0.
\]

To get a feeling for the condition on the left, it is appealing to find
preferably sharp conditions on $a'=(a'_s)_{s\in\N_0}$ and the
distribution of $Z_0$ under which the distance $d_{(\eins)}(\pr
_1^a,\pr_1^{a'})$ does not exceed a given $\delta>0$. This is an
interesting problem on its own, and we will not enlarge upon this here.
We will only mention that, if $\sum_{s=0}^\infty d_{(\eins)}(\mu
_{a_s},\mu_{a'_s})<\infty$ with $\mu_{a_s}$ and $\mu_{a'_s}$ the
laws of $a_sZ_0$ and $a'_sZ_0$, respectively, then, using the
convolution formula, one easily obtains the estimate $d_{(\eins)}(\pr
_1^a,\pr_1^{a'})\le\sum_{s=0}^\infty d_{(\eins)}(\mu_{a_s},\mu
_{a'_s})$ from where on can derive some respective conditions. However,
there might be more sophisticated approaches.
\end{examplenorm}

\begin{examplenorm}[($\operatorname{\mathbf{ARMA}}\boldsymbol{(1,1)}$ processes)]\label{ARMA process}
To illustrate conditions (i)--(iii) in Example~\ref{linear process},
let us consider for any real $\phi_1,\theta_1$ an $\operatorname{ARMA}(1,1)$ process
$X_t^{\phi_1,\theta_1}=\phi_1 X_{t-1}^{\phi_1,\theta_1}
+Z_{t}+\theta_1 Z_{t-1}$ based on a given sequence $(Z_s)_{s\in\Z}$
of square-integrable and centered i.i.d. random variables. Moreover,
let $0<c<1$ be arbitrary but fixed. It is discussed in detail in
Example~\ref{ARMA process - appendix} in the Appendix~\ref{Strong
mixing of linear processes} that if $|\phi_1|,|\theta_1|\le c$ and
$\phi_1\neq-\theta_1$, then the ARMA process $(X_t^{\phi_1,\theta
_1})$ can be represented as a linear process $\sum_{s=0}^\infty
a_s(\phi_1,\theta_1) Z_{t-s}$ with $a_0(\phi_1,\theta_1)=1$.
Letting $A'_{c}:=\{(a_s(\phi_1,\theta_1))_{s\in\N_0}\dvtx |\phi
_1|,|\theta_1|\le c\mbox{ and }\phi_1\neq-\theta_1\}$, Example~\ref{ARMA process - appendix} also shows that
\begin{enumerate}[(iii)]
\item[(i)] $\sum_{s=0}^\infty a_sz^s\neq0$ for all $z$ with $|z|\le
1$, for every $a\in A'_{c}$,
\item[(ii)] $\sup_{a\in A'_{c}}\sum_{u=n}^\infty\sum_{s=u}^\infty
|a_s|\le(2/(1-c)^{2}) c^{n}$ for all $n\in\N$,
\item[(iii)] $\sup_{a\in A'_{c}}\sum_{s=0}^\infty|b_s(a)|\le1+2c/(1-c)$.
\end{enumerate}
Thus, denoting by $\pr^{\phi_1,\theta_1}$ the law of the $\operatorname{ARMA}(1,1)$
process $(X_t^{\phi_1,\theta_1})$ based on the given noise $(Z_s)$
and with coefficients $\phi_1,\theta_1$, the discussion in Example~\ref{linear process} shows that the class $\mathcal{P}=\mathcal
{P}_{c}:=\{\pr^{\phi_1,\theta_1}\dvtx |\phi_1|,|\theta_1|\le c\mbox{
and }\phi_1\neq-\theta_1\}$ satisfies condition (c) in Theorem~\ref
{glivenko cantelli}, because $\mathcal{P}_{c}$ is nothing but the
class of laws $\pr^a$ on $\R^{\N}$ of all linear processes $\sum_{s=0}^\infty a_s Z_{t-s}$ with $a\in A_c'$.
\end{examplenorm}

Recall from \cite{Huber1981}, page 25, that the L\'{e}vy metric
%
\begin{equation}
\label{def levy metric} d_{\mbox{\scriptsize{{L\'{e}vy}}}}(\mu,\nu) := \inf\bigl\{\varepsilon>0 \dvtx
F_\mu (x-\varepsilon)-\varepsilon\le F_\nu(x)\le
F_\mu(x+\varepsilon )+\varepsilon\mbox{ for all }x\in\R\bigr\}
\end{equation}
(with $F_\mu,F_\nu$ the distribution functions of $\mu,\nu\in
\mathcal{M}_1(\R)$) generates the weak topology on $\mathcal{M}_1(\R
)$. Since the classical Kolmogorov metric $d_{(\eins)}$ dominates the
L\'{e}vy metric (cf. \cite{Huber1981}, page 34) and condition (b) in
Theorem~\ref{glivenko cantelli} is always fulfilled if $d_{(\phi
)}=d_{(\eins)}$ (recall Remark~\ref{remark on cond b}(i)), we
immediately obtain the following corollary to Theorem~\ref{glivenko
cantelli} with $\phi:=\eins$.

\begin{corollary}[(UGC property w.r.t. $\boldsymbol{d}_{\scriptsize{\textbf{{L\'{e}vy}}}}$)]\label{glivenko cantelli - corollary}
Let $\mathcal{P}$ be a class of probability measures on $(\Omega
,\mathcal{F})$ satisfying conditions \textup{(a)} and \textup{(c}) of Theorem~\ref
{glivenko cantelli}. Then we have (\ref{glivenko cantelli - eq}) with
$d_{(\phi)}$ replaced by $d_{\mbox{\scriptsize{{L\'{e}vy}}}}$.
\end{corollary}

Obviously, Example~\ref{linear process} remains valid (except the
estimate of $d_{(\eins)}(\pr_1^a,\pr_1^{a'})$ at the end) when
replacing the Kolmogorov metric $d_{(\eins)}$ by the L\'{e}vy metric
$d_{\mbox{\scriptsize{{L\'{e}vy}}}}$.


\subsection{Strong mixing, and \texorpdfstring{$\psi$}{psi}-weak topology}\label{UGC psi
weak topology}

Let $E$ be a locally compact and second countable Hausdorff space (in
this case $E$ is in particular a Polish space), $\mathcal{E}$ be the
corresponding Borel $\sigma$-field, and $\psi\dvtx  E\to[0,\infty)$ be a
continuous function satisfying $\psi\ge1$ outside some compact set.
Let $\mathcal{M}_1^\psi(E)$ be the set of all probability measures
$\mu$ on $(E,\mathcal{E})$ satisfying $\int\psi \,\mathrm{d}\mu<\infty$, and
$C_\psi(E)$ be the space of all continuous functions on $E$ for which
$\|f/(1+\psi)\|_\infty<\infty$, where $\|g\|_\infty:=\sup_{y\in
E}|g(y)|$ for any function $g\dvtx E\to\R$. The \emph{$\psi$-weak
topology} on $\mathcal{M}_1^\psi(E)$ is the coarsest topology for
which the mappings $\mu\mapsto\int f \,\mathrm{d}\mu$, $f\in C_\psi(E)$, are
continuous; cf. Section A.6 in \cite{FoellmerSchied2004}. Clearly, the
$\psi$-weak topology is finer than the weak topology, and the two
topologies coincide if and only if $\psi$ is bounded. It follows from
Lemma~3.4(i)\,$\Leftrightarrow$\,(iii) in \cite{Kraetschmeretal2012a}
(which still holds when replacing $\R^d$ by some Polish space) that
the metric
%
\begin{equation}
\label{def psi weak metric} d_{\psi,\mathrm{{vag}}}(\mu,\nu) := d_{\mathrm{{vag}}}(\mu,\nu )+ \biggl
\llvert \int\psi \,\mathrm{d}\mu-\int\psi \,\mathrm{d}\nu \biggr\rrvert , \qquad\mu,\nu\in\mathcal
{M}_1^\psi(E)
\end{equation}
metrizes the $\psi$-weak topology when $d_{\mathrm{{vag}}}$ metrizes
the vague topology on $\mathcal{M}_1^\psi(E)$. For $d_{\mathrm
{{vag}}}$ we may and do choose
%
\begin{equation}
\label{def vague metric} d_{\mathrm{{vag}}}(\mu,\nu) := \sum_{k=1}^\infty
\frac
{1}{2^{k}} \biggl\{1\wedge \biggl|\int f_k \,\mathrm{d}\mu-\int f_k
\,\mathrm{d}\nu \biggr| \biggr\},\qquad \mu,\nu\in\mathcal{M}_1^\psi(E)
\end{equation}
for some countable and $\|\cdot\|_\infty$-dense subset $\{f_k\}_{k\in
\N}$ of the space $C_{\mathsf c}(E)$ of all continuous functions on
$E$ with compact support; cf. the proof of Theorem~31.5 in \cite{Bauer2001}.

\begin{remarknorm}
Any locally compact and second countable Hausdorff space $E$ is $\sigma
$-compact (cf. Example~2 in Section~29 of \cite{Bauer2001}), that is,
there exists a sequence $(K_n)$ of compact subsets of $E$ such that
$K_n\uparrow E$ and every compact set $K$ is contained in finally all
$K_n$. So by Urysohn's lemma, one can find functions $e_n\in C_{\mathsf
c}(E)$, $n\in\N$, such that $0\le e_n\le1$ and $e_n=1$ on $K_n$ for
all $n\in\N$. If $\{\widetilde{f}_l\}_{l\in\N}$ denotes any
countable and $\|\cdot\|_{\infty}$-dense subset of $C_{\mathsf
c}(E)$, then the set $\{f_k\}_{k\in\N}$ in (\ref{def vague metric})
can be chosen as
\[
\{f_k\}_{k\in\N} := \{\widetilde{f}_l
\}_{l\in\N}\cup\{\widetilde {f}_l e_n
\}_{l,n\in\N}\cup\{e_n\}_{n\in\N}.
\]
This gets clear from the elaborations in the proof of Theorem~31.5 in
\cite{Bauer2001}.
\end{remarknorm}

The co-variance functional $T(\mu):=\int_{\R^2}(x_1-\int_\R x \mu
_{1}(\mathrm{d}x))(x_2-\int_\R x \mu_{2}(\mathrm{d}x)) \mu (\mathrm{d}(x_1,x_2))$
($\mu_{1}$ and $\mu_{2}$ denote the marginal distributions of $\mu
$), for instance, is clearly $\psi$-weakly continuous for $\psi
(x):=|x|^2$ and $E=\R^2$. So, in view of the Hampel-type Theorem~\ref
{hampel-huber generalized}, for a given class $\mathcal{P}$ of
probability measures on $(\Omega,\mathcal{F})$, the corresponding
sequence $(\widehat T_n)$ of plug-in estimators, that is, the sequence
of the sample co-variances, is qualitatively $\mathcal{P}$-marginally
robust w.r.t. $(d_{\psi,\mathrm{{vag}}},d_{\mathrm{{Proh}}}')$ at
any $\pr\in\mathcal{P}$ if $\mathcal{P}$ admits the UGC property
w.r.t. $d_{\psi,\mathrm{{vag}}}$ in the sense of Definition~\ref{def
UGC property}. The following Theorem~\ref{glivenko cantelli - weakl
topology} shows that the latter is true under similar conditions as
imposed in Theorem~\ref{glivenko cantelli}.

\begin{theorem}[(UGC property w.r.t. $\bolds{d}_{\boldsymbol{\psi}\boldsymbol{,}\mathbf
{{vag}}}$)]\label{glivenko cantelli - weakl topology}
Let $E$ be a locally compact and second countable Hausdorff space,
$\mathcal{E}$ be the corresponding Borel $\sigma$-field, and $\psi
\dvtx E\to[0,\infty)$ be a continuous function satisfying $\psi\ge1$
outside some compact set. Let $\mathcal{P}$ be a class of probability
measures on $(\Omega,\mathcal{F})$ such that
\begin{enumerate}[(b)]
\item[(a)] $\pr_1=\pr_2=\cdots$ for all $\pr\in\mathcal{P}$,
\item[(b)] $\lim_{K\to\infty}\sup_{\pr\in\mathcal{P}}\int\psi
(y)\eins_{\psi(y)\ge K} \pr_1(\mathrm{d}y)=0$,
%
\item[(c)] $\lim_{n\to\infty}\sup_{\pr\in\mathcal{P}}\alpha
_\pr(n)=0$.
\end{enumerate}
Then, for every $\delta>0$,
%
\begin{equation}
\label{glivenko cantelli - weakl topology - eq} \lim_{n\to\infty} \sup_{\pr\in\mathcal{P}} \pr
\bigl[ d_{\psi
,\mathrm{{vag}}}(\widehat m_{n},\pr_1)\ge\delta
\bigr] = 0.
\end{equation}
\end{theorem}

The proof of Theorem~\ref{glivenko cantelli} can be found in Section~\ref{proof of glivenko cantelli - weakl topology}. Notice that (c)
implies in particular that the coordinate process $(X_i)$ is strongly
mixing under every $\pr\in\mathcal{P}$.

\begin{remarknorm}\label{remark on cond b - psi}
(i) Notice that condition (b) is always fulfilled if $\psi$ is
bounded, that is, if $d_{\psi,\mathrm{{vag}}}$ metrizes the classical
weak topology.

(ii) In the case $E=\R$, condition (b) is fulfilled if there is some
$w\dvtx \R\to[0,\infty)$ such that $\lim_{|x|\to\infty}w(x)/\phi
(x)=\infty$ and $\sup_{\pr\in\mathcal{P}}\int w(y) \pr
_1(\mathrm{d}y)<\infty$; cf. \cite{Kraetschmeretal2012b}, Lemma~4.1.
\end{remarknorm}

Recall from (\ref{def levy metric}) the definition of the L\'{e}vy
metric $d_{\mbox{\scriptsize{{L\'{e}vy}}}}$. Since $d_{\mbox{\scriptsize{{L\'{e}vy}}}}$
generates the weak topology on $\mathcal{M}_1(\R)$, the metric
%
\begin{equation}
\label{def levy psi metric} d_{\psi,\mbox{\scriptsize{{L\'{e}vy}}}}(\mu,\nu) := d_{\mbox{\scriptsize{{L\'{e}vy}}}}(\mu,\nu)+ \biggl
\llvert \int\psi \,\mathrm{d}\mu-\int\psi \,\mathrm{d}\nu \biggr\rrvert ,\qquad \mu ,\nu\in
\mathcal{M}_1^\psi(\R)
\end{equation}
generates the $\psi$-weak topology on $\mathcal{M}_1^\psi(\R)$. The
following corollary is an immediate consequence of Theorem~\ref
{glivenko cantelli - weakl topology} (in fact of (\ref{glivenko
cantelli - weakl topology - eq - proof 2}) in its proof) and Corollary~\ref{glivenko cantelli - corollary}.

\begin{corollary}[(UGC property w.r.t. $\bolds{d}_{\boldsymbol{\psi}\boldsymbol{,}\scriptsize{\textbf{{L\'{e}vy}}}}$)]
Let $E=\R$, and $\mathcal{P}$ be a class of probability measures on
$(\Omega,\mathcal{F})$ satisfying conditions \textup{(a)--(c)} of Theorem~\ref
{glivenko cantelli - weakl topology}. Then we have (\ref{glivenko
cantelli - weakl topology - eq}) with $d_{\psi,\mathrm{{vag}}}$
replaced by $d_{\psi,\mbox{\scriptsize{{L\'{e}vy}}}}$.
\end{corollary}

\begin{examplenorm}
(i) It was already mentioned in Example~\ref{remark on cond b - role
of phi}(i) that for fixed $\alpha\in(0,1)$ the lower $\alpha
$-quantile functional is continuous at $\mu\in\mathcal{M}_1(\R)$
w.r.t. the classical weak topology, that is, w.r.t. the $\psi$-weak
topology with $\psi=\eins$, provided $F^{\leftarrow}_{\mu}$ is
continuous at $\alpha$.

(ii) It is easily seen that the mean functional and the variance
functional are continuous w.r.t. the $\psi$-weak topology on $\mathcal
{M}_1^\psi(\R)$ for $\psi=|x|$ and $\psi=|x|^2$, respectively.

(iii) It is demonstrated in \cite{Kraetschmeretal2012b} that the
statistical functional $T(\mu)=\rho(X_\mu)$ corresponding to any
law-invariant convex risk measure $\rho$ defined on an Orlicz space
with continuous Young function $\Psi$ satisfying the $\Delta
_2$-condition (meaning that there are $C,x_0>0$ such that $\Psi(2x)\le
C\Psi(x)$ for all $x\ge x_0$) is continuous w.r.t. the $\psi$-weak
topology on $\mathcal{M}_1^\psi(\R)$ for $\psi(\cdot)=\Psi(|\cdot|)$.
\end{examplenorm}


\section{Proof of Theorem \texorpdfstring{\protect\ref{hampel-huber generalized}}{2.4}}\label
{proof of hampel-huber generalized}

We adapt the proof of the classical Hampel theorem as given in \cite
{Huber1981,HuberRonchetti2009}. For the reader's convenience, we first
of all recall Strassen's theorem (as formulated in Theorem~2.4.7 in
\cite{Huber1981}; the proof is contained in the seminal paper \cite
{Strassen1965}) whose implication (ii)$\,\Rightarrow\,$(i) is the key for
the proof of Theorem~\ref{hampel-huber generalized}.

\begin{theorem}[(Strassen)]\label{strassens theorem}
Let $\T$ be a Polish space equipped with the corresponding Borel
$\sigma$-field $\mathcal{T}$, and $d_\T$ be any complete and
separable metric generating the topology on $\T$. Then, for any two
probability measures $\mu_1,\mu_2$ on $(\T,\mathcal{T})$ and any
$\varepsilon,\delta>0$, the following two statements are equivalent:
\begin{enumerate}[(ii)]
\item[(i)] For every $A\in\mathcal{T}$ we have
\[
\mu_1[A] \le\mu_2\bigl[A^\delta\bigr]+
\varepsilon,
\]
where $A^\delta:=\{t\in\T\dvtx  \inf_{a\in A}d_\T(t,a)\le\delta\}$.
\item[(ii)] There is some probability measure $\mu$ on $(\T\times\T
,\mathcal{T}\times\mathcal{T})$ such that $\mu\circ\pi_1^{-1}=\mu
_1$, $\mu\circ\pi_2^{-1}=\mu_2$, and
\[
\mu \bigl[ \bigl\{(t_1,t_2)\in\T\times\T\dvtx
d_\T(t_1,t_2)\le\delta \bigr\} \bigr] \ge1-
\varepsilon,
\]
where $\pi_i\dvtx \T\times\T\to\T$ denotes the projection on the $i$th
coordinate, $i=1,2$.
\end{enumerate}
\end{theorem}

To prove Theorem~\ref{hampel-huber generalized}, we have to show that
for every $\varepsilon>0$ there are some $\delta>0$ and $n_0\in\N$
such that
%
\begin{equation}
\label{proof of hampel-huber generalized - eq - 1} \Q\in\mathcal{P},\qquad d(\pr_1,\Q_1)\le
\delta\quad\Longrightarrow\quad d_{\mathrm{{Proh}}}'\bigl(\pr\circ
\widehat T_n^{-1},\Q\circ\widehat T_n^{-1}
\bigr)\le\varepsilon\qquad\forall n\ge n_0.
\end{equation}
Since
\[
d_{\mathrm{{Proh}}}'\bigl(\pr\circ\widehat T_n^{-1},
\Q\circ\widehat T_n^{-1}\bigr) \le d_{\mathrm{{Proh}}}'
\bigl(\pr\circ\widehat T_n^{-1},\delta_{T(\pr_1)}\bigr)
+ d_{\mathrm{{Proh}}}'\bigl(\delta_{T(\pr_1)},\Q\circ\widehat
T_n^{-1}\bigr)
\]
(with $\delta_{T(\pr_1)}$ the dirac measure on $(\T,\mathcal{T})$
with atom $T(\pr_1)$), for (\ref{proof of hampel-huber generalized -
eq - 1}) it suffices to show that for every $\varepsilon>0$ there are
some $\delta>0$ and $n_0\in\N$ such that
%
\begin{equation}
\label{proof of hampel-huber generalized - eq - 2} \Q\in\mathcal{P},\qquad d(\pr_1,\Q_1)\le
\delta\quad\Longrightarrow\quad d_{\mathrm{{Proh}}}'\bigl(
\delta_{T(\pr_1)},\Q\circ\widehat T_n^{-1}\bigr)\le
\varepsilon/2 \qquad\forall n\ge n_0.
\end{equation}
The remainder of the proof is divided into two steps. In Step 1, we
will verify that for (\ref{proof of hampel-huber generalized - eq -
2}) it suffices to show that for every $\varepsilon>0$ there are some
$\delta>0$ and $n_0\in\N$ such that
%
\begin{eqnarray}
\label{proof of hampel-huber generalized - eq - 3}
\Q\in\mathcal{P},\qquad d(\pr_1,\Q_1)\le
\delta&\quad\Longrightarrow \quad&\Q \biggl[ \biggl\{x\in\Omega\dvtx
d_\T\bigl(T(\pr_1),\widehat T_n(x)\bigr)\le
\frac
{\varepsilon}{2} \biggr\} \biggr]\ge1-\frac{\varepsilon}{2}
\nonumber
\\[-8pt]
\\[-8pt]
&&\quad\forall n\ge
n_0.\nonumber
\end{eqnarray}
In Step 2, we will verify (\ref{proof of hampel-huber generalized - eq
- 3}).

\emph{Step} 1. We note that the right-hand side in (\ref{proof of
hampel-huber generalized - eq - 3}) is equivalent to
%
\begin{equation}
\label{proof of hampel-huber generalized - eq - 3 - 5} \bigl(\delta_{T(\pr_1)}\times\bigl(\Q\circ\widehat
T_n^{-1}\bigr) \bigr) \biggl[ \biggl\{(t_1,t_2)
\in\T\times\T\dvtx d_\T(t_1,t_2)\le
\frac
{\varepsilon}{2} \biggr\} \biggr]\ge1-\frac{\varepsilon}{2} \qquad\forall n\ge
n_0.
\end{equation}
In view of the implication (ii)$\,\Rightarrow\,$(i) in Strassen's Theorem~\ref{strassens theorem} (with $\mu:=\delta_{T(\pr_1)}\times(\Q
\circ\widehat T_n^{-1})$ and $\widetilde\varepsilon:=\widetilde
\delta:=\varepsilon/2$), condition (\ref{proof of hampel-huber
generalized - eq - 3 - 5}) implies
\[
\delta_{T(\pr_1)}[A] \le\Q\circ\widehat T_n^{-1}
\bigl[A^{\varepsilon
/2}\bigr]+\varepsilon/2 \qquad\forall A\in\mathcal{T}, n\ge
n_0,
\]
that is
\[
d_{\mathrm{{Proh}}}'\bigl(\delta_{T(\pr_1)},\Q\circ\widehat
T_n^{-1}\bigr) \le\varepsilon/2 \qquad\forall n\ge
n_0.
\]
That is, the right-hand side in (\ref{proof of hampel-huber
generalized - eq - 3}) implies the right-hand side in (\ref{proof of
hampel-huber generalized - eq - 2}).

\emph{Step} 2. To verify (\ref{proof of hampel-huber generalized - eq
- 3}), we pick $\varepsilon>0$. Since $T$ is $(d,d_\T)$-continuous at
$\pr_1$, we can find some $\delta>0$ such that for every $x\in\Omega
$ and $n\in\N$
%
\begin{equation}
\label{proof of hampel-huber generalized - eq - 4} d\bigl(\pr_1,\widehat m_n(x)\bigr)\le2
\delta\quad\Longrightarrow\quad d_\T\bigl(T(\pr _1),T
\bigl(\widehat m_n(x)\bigr)\bigr)\le\varepsilon/2;
\end{equation}
recall that the class of all empirical probability measures $\widehat
m_n(x)$ was assumed to be contained in~$\mathcal{M}$. Now, for any $\Q
\in\mathcal{P}$ satisfying $d(\pr_1,\Q_1)\le\delta$, we obtain
with the help of the UGC property of $\mathcal{P}$ w.r.t. $d$, of the
implication
\[
d\bigl(\Q_1,\widehat m_n(x)\bigr)\le\delta\quad
\Longrightarrow\quad d\bigl(\pr_1,\widehat m_n(x)\bigr)
\bigl[\le d(\pr_1,\Q_1)+d\bigl(\Q_1,\widehat
m_n(x)\bigr)\bigr] \le d(\pr_1,\Q _1)+\delta
\]
and of (\ref{proof of hampel-huber generalized - eq - 4}) that
\begin{eqnarray*}
1-\frac{\varepsilon}{2} & \le& \Q \bigl[ \bigl\{x\in\Omega\dvtx d\bigl(
\Q_1,\widehat m_n(x)\bigr)\le \delta \bigr\} \bigr]
\\
& \le& \Q \bigl[ \bigl\{x\in\Omega\dvtx d\bigl(\pr_1,\widehat
m_n(x)\bigr)\le d(\pr_1,\Q_1)+\delta \bigr\}
\bigr]
\\
& \le& \Q \bigl[ \bigl\{x\in\Omega\dvtx d\bigl(\pr_1,\widehat
m_n(x)\bigr)\le 2\delta \bigr\} \bigr]
\\
& \le& \Q \bigl[ \bigl\{x\in\Omega\dvtx d_\T\bigl(T(
\pr_1),T\bigl(\widehat m_n(x)\bigr)\bigr)\le\varepsilon/2
\bigr\} \bigr]
\\
& = & \Q \bigl[ \bigl\{x\in\Omega\dvtx d_\T\bigl(T(
\pr_1),\widehat T_n(x)\bigr)\le \varepsilon/2 \bigr\}
\bigr]
\end{eqnarray*}
for all $n\ge n_0$ and some $n_0=n_0(\varepsilon)\in\N$, where $n_0$
can be chosen independently of $\Q$ by the UGC property of $\mathcal
{P}$. Thus, (\ref{proof of hampel-huber generalized - eq - 3}) holds.
This completes the proof of Theorem~\ref{hampel-huber generalized}.


\section{Proof of Theorem \texorpdfstring{\protect\ref{glivenko cantelli}}{3.1}}\label{proof of
glivenko cantelli}

We will only show that for every $\delta>0$
%
\begin{equation}
\label{appendix a2 - eq 0} \lim_{n\to\infty} \sup_{\pr\in\mathcal{P}} \pr
\Bigl[\sup_{y\in
(-\infty,0]}\bigl\llvert \widehat F_n(y)-F_{\pr}(y)
\bigr\rrvert \phi(y) \ge\delta \Bigr] = 0,
\end{equation}
where $\widehat F_n$ and $F_{\pr}$ denote the empirical distribution
function of $X_1,\ldots,X_n$ and the distribution function of the
marginal distribution $\pr_1$, respectively. The analogous result for
the positive real line can be shown in the same way. We will proceed in
five steps. For every $\pr\in\mathcal{P}$, let the functions $w_\pr
\dvtx [0,1]\to[0,\infty]$ and $h_\pr\dvtx [0,1]\to[0,\infty)$ be defined by
\begin{eqnarray*}
w_\pr(t) & := & \phi\bigl(F_{\pr}^\leftarrow(t)\bigr)
\eins_{[0,F_{\pr
}(0)]}(t),
\\
h_\pr(t) & := & \int_0^tw_\pr(s)
\,\mathrm{d}s.
\end{eqnarray*}
Of course, the functions $h_\pr$ are continuous, nondecreasing and
satisfy $h_\pr(0)=0$ for all $\pr\in\mathcal{P}$. By assumption
(b), we also have $\sup_{\pr\in\mathcal{P}}h_\pr(1) (\le\sup_{\pr\in\mathcal{P}}\int\phi \,\mathrm{d}\pr_1)<\infty$.\vadjust{\goodbreak}

\emph{Step} 1. As the first step, we will show that the family of
functions $\{h_\pr\}_{\pr\in\mathcal{P}}$ is uniformly
equicontinuous on $[0,1]$. Denote by $J_\pr$ the set of all $y\in
(-\infty,0]$ at which $F_\pr$ has a jump, and by $I_\pr\subset
[0,1]$ the range of $F_\pr$. Noting that $F_\pr(F_\pr^\leftarrow
(s))=s$ if and only if $s\in I_\pr$ and that $F_\pr^\leftarrow(s)\in
J_\pr$ for all $s\in[0,1]\setminus I_\pr$ (cf. \cite{Shorack2000}, page
113), we obtain for any $K>0$\looseness=1\vspace*{-1pt}
\begin{eqnarray*}
&&{\sup_{\pr\in\mathcal{P}} h_\pr(t)}
\\[-1pt]
&&\quad = \sup_{\pr\in\mathcal{P}} \int\phi\bigl(F_{\pr}^\leftarrow
(s)\bigr)\eins_{[0,F_{\pr}(0)\wedge t]\cap I_\pr}(s) \,\mathrm{d}s
\\[-1pt]
&&\qquad{} + \sup_{\pr\in\mathcal{P}} \int\phi\bigl(F_{\pr}^\leftarrow
(s)\bigr)\eins_{[0,F_{\pr}(0)\wedge t]\setminus I_\pr}(s) \,\mathrm{d}s
\\[-1pt]
&&\quad \le \sup_{\pr\in\mathcal{P}} \int\phi\bigl(F_{\pr}^\leftarrow
(s)\bigr)\eins_{[0,F_{\pr}(0)\wedge t]\cap I_\pr}\bigl(F_\pr\bigl(F_\pr
^\leftarrow(s)\bigr)\bigr) \,\mathrm{d}s
\\[-1pt]
&&\qquad{} + \sup_{\pr\in\mathcal{P}} \sum_{y\in J_\pr}
\phi(y) \bigl(\bigl(F_\pr(y)\wedge t\bigr)-\bigl(F_\pr(y-)
\wedge t\bigr) \bigr)
\\[-1pt]
&&\quad \le \sup_{\pr\in\mathcal{P}} \int_{\R_-}\phi(y)
\eins_{\phi
(y)\ge K} \pr_1(\mathrm{d}y) + \sup_{\pr\in\mathcal{P}}
K\int_{\R_-}\eins _{[0,t]}\bigl(F_\pr(y)
\bigr) \pr_1(\mathrm{d}y)
\\[-1pt]
&&\qquad{} + \sup_{\pr\in\mathcal{P}} \int_{\R_-}\phi(y)
\eins_{\phi
(y)\ge K} \pr_1(\mathrm{d}y) + \sup_{\pr\in\mathcal{P}}
K\sum_{y\in J_\pr
} \bigl(\bigl(F_\pr(y)\wedge t
\bigr)-\bigl(F_\pr(y-)\wedge t\bigr) \bigr)
\\[-1pt]
&&\quad \le \sup_{\pr\in\mathcal{P}} \int\phi(y) \eins_{\phi(y)\ge
K}
\pr_1(\mathrm{d}y) + \sup_{\pr\in\mathcal{P}} K \pr_1
\bigl[\bigl\{z\dvtx F_\pr(z)\le t\bigr\}\bigr]
\\[-1pt]
&&\qquad{} + \sup_{\pr\in\mathcal{P}} \int\phi(y) \eins_{\phi(y)\ge K}
\pr_1(\mathrm{d}y) + K t
\\[-1pt]
&&\quad = 2 \sup_{\pr\in\mathcal{P}} \int\phi(y) \eins_{\phi(y)\ge
K}
\pr_1(\mathrm{d}y) + 2K t.
\end{eqnarray*}
Now, by assumption (b) we may choose $K=K_\varepsilon>0$ such that the
first summand is bounded above by $\varepsilon/2$. In particular,
$h_\pr(t)\le\varepsilon$ for all $t\le\varepsilon/(4K_\varepsilon
)$ and $\pr\in\mathcal{P}$. That is, $h_\pr$ is (right) continuous
at $0$ uniformly in $\pr\in\mathcal{P}$. The uniform (right)
continuity at $0$ moreover implies uniform equicontinuity of the family
$\{h_\pr\}_{\pr\in\mathcal{P}}$ on $[0,1]$, because for every $t\in
[0,1]$ and $\Delta\in\R$ with $t+\Delta\in[0,1]$,\vspace*{-1pt}
%
\begin{eqnarray}
\label{equicontinuity} \sup_{\pr\in\mathcal{P}}\bigl\llvert h_\pr(t)-h_\pr(t+
\Delta)\bigr\rrvert & = & \sup_{\pr\in\mathcal{P}} \biggl\llvert \int
_{t+(\Delta\wedge
0)}^{t+(\Delta\vee0)}\phi\bigl(F_{\pr}^\leftarrow(s)
\bigr)\eins_{[0,F_{\pr
}(0)]}(s) \,\mathrm{d}s\biggr\rrvert
\nonumber
\\[-1pt]
& \le& \sup_{\pr\in\mathcal{P}}\int_0^{|\Delta|}
\phi\bigl(F_{\pr
}^\leftarrow(s)\bigr)\eins_{[0,F_{\pr}(0)]}(s) \,\mathrm{d}s
\\[-1pt]
& = & \sup_{\pr\in\mathcal{P}} h_{\pr}\bigl(|\Delta|\bigr),
\nonumber
\end{eqnarray}
where we used the fact that $\phi(F_{\pr}^\leftarrow(\cdot))\eins
_{[0,F_{\pr}(0)]}(\cdot)$ is nonincreasing on $[0,1]$.\vadjust{\goodbreak}

\emph{Step} 2. Next, we prepare for Step 3. On the one hand, by the
uniform equicontinuity of the family $\{h_\pr\}_{\pr\in\mathcal
{P}}$ on the compact interval $[0,1]$ (cf. (\ref{equicontinuity}) in
Step 1), we can find for every $\varepsilon>0$ a finite partition
$0={s_0^\varepsilon}<{s_1^\varepsilon}<\cdots<{s_{k_\varepsilon
}^\varepsilon}=1$ (being independent of $\pr$) such that
\[
\sup_{\pr\in\mathcal{P}} \sup_{i=1,\ldots,k_\varepsilon} \int
_{s_{i-1}^\varepsilon}^{s_i^\varepsilon} w_\pr(s) \,\mathrm{d}s \Bigl(= \sup
_{\pr\in\mathcal{P}} \sup_{i=1,\ldots,k_\varepsilon} \bigl(h_\pr
\bigl(s_i^\varepsilon\bigr)-h_\pr\bigl(s_{i-1}^\varepsilon
\bigr)\bigr) \Bigr) \le\varepsilon/2.
\]
On the other hand, by assumption (b) we may choose a constant
$K_\varepsilon>0$ such that $\sup_{\pr\in\mathcal{P}}\int\phi
(z)\eins_{\phi(z)\ge K_\varepsilon}\pr_1(\mathrm{d}z)\le\varepsilon/2$.
Thus, noting
\begin{eqnarray*}
w_\pr^{\rightarrow}(y) & = & \sup\bigl\{s\in[0,1] \dvtx \phi
\bigl(F_\pr^\leftarrow(s)\bigr)\eins_{[0,F_\pr
(0)]}(s)>y\bigr\}
\\
& \le& \sup\bigl\{s\in[0,1] \dvtx s\le F_\pr\bigl(
\phi^\rightarrow(y)\bigr)\wedge F_\pr(0)\bigr\}
\\
& \le& F_\pr\bigl(\phi^\rightarrow(y)\bigr)
\end{eqnarray*}
for $y\in(0,\infty)$, and using integration-by-parts (more precisely
Theorem~1.15 in \cite{Mattila1995}), we obtain
\begin{eqnarray*}
\sup_{\pr\in\mathcal{P}} \int_{K_\varepsilon}^\infty
w_\pr ^\rightarrow(y) \,\mathrm{d}y & \le& \sup_{\pr\in\mathcal{P}} \int
_{K_\varepsilon}^\infty F_\pr\bigl(
\phi^\rightarrow(y)\bigr) \,\mathrm{d}y
\\
& = & \sup_{\pr\in\mathcal{P}} \biggl(\int_{(-\infty,\phi
^\rightarrow(K_\varepsilon)]}\phi(z)
\pr_1(\mathrm{d}z)-K_\varepsilon F_{\pr
}\bigl(
\phi^\rightarrow(K_\varepsilon)\bigr) \biggr)
\\
& \le& \sup_{\pr\in\mathcal{P}} \int_{(-\infty,\phi^\rightarrow
(K_\varepsilon)]}\phi(z)
\pr_1(\mathrm{d}z)
\\
& \le& \sup_{\pr\in\mathcal{P}} \int_{(-\infty,0]}\phi(z) \eins
_{\phi(z)\ge K_\varepsilon} \pr_1(\mathrm{d}z)
\\
& \le& \varepsilon/2,
\end{eqnarray*}
where $\phi^\rightarrow(z):=\sup\{y\in(-\infty,0]\dvtx \phi(y)>z\}$,
$z\in[0,\infty)$, denotes the right continuous inverse of the
nonincreasing function $\phi\dvtx (-\infty,0]\to[1,\infty)$. Taking into
account that the functions $w_\pr^\rightarrow$ take values only in
$[0,1]$, we can find a finite partition $0={y_0^\varepsilon
}<{y_1^\varepsilon}<\cdots<{y_{l_\varepsilon-1}^\varepsilon
}<{y_{l_\varepsilon}^\varepsilon}=\infty$ (being independent of $\pr
$), with ${y_{l_\varepsilon-1}^\varepsilon}=K_\varepsilon$, such that
\[
\sup_{\pr\in\mathcal{P}} \sup_{i=1,\ldots,l_\varepsilon} \int
_{y_{i-1}^\varepsilon}^{y_{i}^\varepsilon} w_\pr^\rightarrow(y) \,\mathrm{d}y
\le\varepsilon/2,
\]
that is, in other words,
\[
\sup_{\pr\in\mathcal{P}} \sup_{i=1,\ldots,l_\varepsilon} \int
_0^1 \bigl(\bigl(y_{i}^\varepsilon
\wedge w_\pr(s)\bigr)-\bigl(y_{i-1}^\varepsilon \wedge
w_\pr(s)\bigr) \bigr)\,\mathrm{d}s \le\varepsilon/2.
\]
Finally, for every $\pr\in\mathcal{P}$, we let $0=t_{0,\pr
}^\varepsilon<t_{1,\pr}^\varepsilon<\cdots<t_{m_{\varepsilon,\pr
},\pr}^\varepsilon=1$ be the partition consisting of all points
$s_i^\varepsilon$ and $w_\pr^\rightarrow(y_i^\varepsilon)$, where
$m_{\varepsilon,\pr}\le k_\varepsilon+l_\varepsilon$, and
$k_\varepsilon+l_\varepsilon=:m_\varepsilon$ is independent of $\pr
$. For notational simplicity, we assume without loss of generality
$m_{\varepsilon,\pr}=m_\varepsilon$ for all $\pr\in\mathcal{P}$.

\emph{Step} 3. Let $L^1(\mathrm{d}\I)$ be the space of all Lebesgue integrable
functions on $[0,1]$, and $[l,u]:=\{f\in L^1(\mathrm{d}\I)\dvtx l\le f\le u\}$ be
the bracket of two functions $l,u\in L^1(\mathrm{d}\I)$ with $l\le u$
pointwise. For any $\varepsilon>0$, a bracket $[l,u]$ is called
$\varepsilon$-bracket if $\int_0^1(u-l) \,\mathrm{d}\I<\varepsilon$; cf. \cite
{vanderVaartWellner1996}, page 83. Using the notation introduced in Step
2, we set for every $\pr\in\mathcal{P}$ and $i=1,\ldots
,m_\varepsilon$
\begin{eqnarray*}
l_{i,\pr}^\varepsilon(\cdot) & := & w_\pr
\bigl(t_{i,\pr}^\varepsilon \bigr)\eins_{[0,t_{i-1,\pr}^\varepsilon]}(\cdot),
\\
u_{i,\pr}^\varepsilon(\cdot) & := & w_\pr
\bigl(t_{i-1,\pr}^\varepsilon \bigr)\eins_{[0,t_{i-1,\pr}^\varepsilon]}(\cdot) +
w_\pr(\cdot) \eins _{(t_{i-1,\pr}^\varepsilon,t_{i,\pr}^\varepsilon]}(\cdot).
\end{eqnarray*}
It follows from the choice of the $t_{i,\pr}^\varepsilon$ that, for
every $\pr\in\mathcal{P}$, $[l_{1,\pr}^\varepsilon,u_{1,\pr
}^\varepsilon],\ldots,[l_{m_\varepsilon,\pr}^\varepsilon
,u_{m_\varepsilon,\pr}^\varepsilon]$ provide $\varepsilon$-brackets
in $L^1(\mathrm{d}\I)$ covering the class $\mathcal{E}_\pr:=\{w_{s,\pr} \dvtx
s\in[0,1]\}$ of functions
\[
w_{s,\pr}(\cdot) := w_\pr(s)\eins_{[0,s]}(\cdot).
\]

\emph{Step} 4. By the usual quantile transformation \cite
{ShorackWellner1986}, page 103, we can find for every $\pr\in\mathcal
{P}$ a sequence of $U[0,1]$-random variables $U_{1,\pr},U_{2,\pr
},\ldots$ (possibly on an extension $(\Omega_\pr,\mathcal{F}_\pr
,\overline{\pr})$ of the original probability space $(\Omega
,\mathcal{F},\pr)$) such that the sequence $(U_{i,\pr})$ has the
same mixing coefficients (under $\overline{\pr}$) as the sequence
$(X_i)$ under $\pr$ and such that the corresponding empirical
distribution function $\widehat G_{n,\pr}$ satisfies $\widehat
F_n=\widehat G_{n,\pr}\circ F_{\pr}$ $\overline{\pr}$-almost
surely. Here we will show as in the proof of Theorem~2.4.1 in \cite
{vanderVaartWellner1996} that $\overline{\pr}$-almost surely
%
\begin{eqnarray}
\label
{appendix a2 - eq 1} &&{\sup_{y\le0} \bigl\llvert \widehat
F_n(y)-F_{\pr}(y) \bigr\rrvert \phi (y)}
\nonumber
\\[-8pt]
\\[-8pt]
&&\quad \le \max_{i=1,\ldots,m_\varepsilon}\max \biggl\{\int_0^1
u_{i,\pr}^\varepsilon \,\mathrm{d}(\widehat G_{n,\pr}-\I) ;
\int_0^1 l_{i,\pr
}^\varepsilon \,
\mathrm{d}(\I-\widehat G_{n,\pr}) \biggr\} + \varepsilon
\nonumber
\end{eqnarray}
for every $\varepsilon>0$. Since $F_\pr^\leftarrow(F_\pr(y))\le y$
for all $y\in\R$ (cf. \cite{Shorack2000}, page 113) and $\phi$ is
nonincreasing on $(-\infty,0]$, we have $\overline{\pr}$-almost surely
\begin{eqnarray*}
\sup_{x\le0} \bigl\llvert \widehat F_n(y)-F_{\pr}(y)
\bigr\rrvert \phi(y) & = & \sup_{y\le0} \bigl\llvert \widehat
G_{n,\pr}\bigl(F_\pr(y)\bigr)-F_\pr(y) \bigr
\rrvert \phi(y)
\\
& \le& \sup_{y\le0} \bigl\llvert \widehat G_{n,\pr}
\bigl(F_\pr(y)\bigr)-F_\pr (y) \bigr\rrvert \phi
\bigl(F_\pr^\leftarrow\bigl(F_\pr(y)\bigr)\bigr)
\\
& \le& \sup_{s\in(0,1)}\bigl\llvert \widehat G_{n,\pr}(s)-s
\bigr\rrvert w_\pr(s)
\\
& = & \sup_{s\in(0,1)} \biggl\llvert \int_0^1
w_{s,\pr} \,\mathrm{d}\widehat G_{n,\pr
}-\int_0^1w_{s,\pr}
\,\mathrm{d}\I \biggr\rrvert .
\end{eqnarray*}
So for (\ref{appendix a2 - eq 1}) it suffices to show that
%
\begin{eqnarray}
\label{appendix a2 - eq 1 - alt} &&{\sup_{s\in(0,1)} \biggl\llvert \int
_0^1 w_{s,\pr} \,\mathrm{d}\widehat
G_{n,\pr}-\int_0^1w_{s,\pr} \,\mathrm{d}
\I \biggr\rrvert }
\nonumber
\\[-8pt]
\\[-8pt]
&&\quad \le \max_{i=1,\ldots,m_\varepsilon} \max \biggl\{\int_0^1
u_{i,\pr}^\varepsilon \,\mathrm{d}(\widehat G_{n,\pr}-\I) ;
\int_0^1 l_{i,\pr
}^\varepsilon \,
\mathrm{d}(\I-\widehat G_{n,\pr}) \biggr\} + \varepsilon.
\nonumber
\end{eqnarray}
To prove (\ref{appendix a2 - eq 1 - alt}), we note that for every
$s\in[0,1]$ there is some $i_s=i_s(\pr)\in\{1,\ldots,m_\varepsilon
\}$ such that $w_{s,\pr}\in[l_{i_s,\pr}^\varepsilon,u_{i_s,\pr
}^\varepsilon]$; cf. Step 3. Therefore, since $[l_{i_s;\pr
}^\varepsilon,u_{i_s,\pr}^\varepsilon]$ is an $\varepsilon$-bracket,
\begin{eqnarray*}
\int_0^1 w_{s,\pr} \,\mathrm{d}
\widehat G_{n,\pr}-\int_0^1
w_{s,\pr} \,\mathrm{d}\I & \le& \int_0^1
u_{i_s,\pr}^\varepsilon \,\mathrm{d}\widehat G_{n,\pr}-\int
_0^1 w_{s,\pr} \,\mathrm{d}\I
\\
& = & \int_0^1 u_{i_s,\pr}^\varepsilon
\,\mathrm{{d}}(\widehat G_{n,\pr}-\I )+\int_0^1
\bigl(u_{i_s,\pr}^\varepsilon-w_{s,\pr}\bigr) \,\mathrm{d}\I
\\
& \le& \int_0^1 u_{i_s,\pr}^\varepsilon
\,\mathrm{d}(\widehat G_{n,\pr}-\I )+\int_0^1
\bigl(u_{i_s,\pr}^\varepsilon-l_{i_s,\pr}^\varepsilon\bigr) \,
\mathrm{d}\I
\\
& \le& \max_{i=1,\ldots,m_\varepsilon}\int_0^1
u_{i,\pr
}^\varepsilon \,\mathrm{d}(\widehat G_{n,\pr}-\I) +
\varepsilon.
\end{eqnarray*}
Analogously, we obtain
\[
\int_0^1 w_{s,\pr} \,\mathrm{d}\widehat
G_{n,\pr}-\int_0^1 w_{s,\pr} \,\mathrm{d}
\I \ge - \biggl(\max_{i=1,\ldots,m_\varepsilon}\int_0^1
l_{i,\pr
}^\varepsilon \,\mathrm{d}(\I-\widehat G_{n,\pr}) +
\varepsilon \biggr).
\]
That is, (\ref{appendix a2 - eq 1 - alt}) and therefore (\ref
{appendix a2 - eq 1}) hold true.

\emph{Step} 5. Because of (\ref{appendix a2 - eq 1}), for (\ref
{appendix a2 - eq 0}) to be true it suffices to show that for every
$\delta>0$
%
\begin{equation}
\label{appendix a2 - eq 100} \lim_{n\to\infty} \sup_{\pr\in\mathcal{P}}
\overline{\pr} \biggl[\max_{i=1,\ldots,m_\varepsilon}\max \biggl\{\int
_0^1 u_{i,\pr
}^\varepsilon \,\mathrm{d}(
\widehat G_{n,\pr}-\I) ; \int_0^1
l_{i,\pr
}^\varepsilon \,\mathrm{d}(\I-\widehat G_{n,\pr}) \biggr\} \ge
\delta/2 \biggr] = 0
\end{equation}
with $\varepsilon=\varepsilon(\delta):=\delta/2$. For (\ref
{appendix a2 - eq 100}) to be true it suffices to show that
%
\begin{eqnarray}
\label{glivenko cantelli - eq - proof 1} \lim_{n\to\infty} \sup_{\pr\in\mathcal{P}}
\overline{\pr} \biggl[ \biggl\llvert \int_0^1
l_{i,\pr}^\varepsilon \,\mathrm{d}(\I-\widehat G_{n,\pr})
\biggr\rrvert \ge\delta/2 \biggr] & = & 0,
\\
\label{glivenko cantelli - eq - proof 2} \lim_{n\to\infty} \sup_{\pr\in\mathcal{P}}
\overline{\pr} \biggl[ \biggl\llvert \int_0^1
u_{i,\pr}^\varepsilon \,\mathrm{d}(\widehat G_{n,\pr}-\I)
\biggr\rrvert \ge\delta/2 \biggr] & = & 0,
\end{eqnarray}
for every $i=1,\ldots,m_\varepsilon$, because $\overline{\pr}$ is
subadditive. We will show only (\ref{glivenko cantelli - eq - proof
2}). Assertion (\ref{glivenko cantelli - eq - proof 1}) can be shown
analogously. Since
\begin{eqnarray*}
&&{\int_0^1 u_{i,\pr}^\varepsilon
\,\mathrm{d}(\widehat G_{n,\pr}-\I)}
\\
&&\quad = \frac{1}{n}\sum_{j=1}^n
\bigl(w_\pr\bigl(t_{i-1,\pr}^\varepsilon \bigr)
\eins_{[0,t_{i-1,\pr}^\varepsilon]}(U_{j,\pr})-\ex_{\overline
{\pr}} \bigl[w_\pr
\bigl(t_{i-1,\pr}^\varepsilon\bigr) \eins_{[0,t_{i-1,\pr
}^\varepsilon]}(U_{1,\pr})
\bigr] \bigr)
\\
&&\qquad {} + \frac{1}{n}\sum_{j=1}^n
\bigl(w_\pr(U_{j,\pr})\eins _{(t_{i-1,\pr}^\varepsilon,t_{i,\pr}^\varepsilon]}(U_{j,\pr})-
\ex _{\overline{\pr}} \bigl[w_\pr(U_{1,\pr})
\eins_{(t_{i-1,\pr
}^\varepsilon,t_{i,\pr}^\varepsilon]}(U_{1,\pr}) \bigr] \bigr),
\end{eqnarray*}
for (\ref{glivenko cantelli - eq - proof 2}) to be true it suffices to
show that for every fixed $i\in\{1,\ldots,m_\varepsilon\}$
%
\begin{eqnarray}
\label{glivenko cantelli - eq - proof 2 - 1} &&
\hspace*{2pt}\lim_{n\to\infty} \sup_{\pr\in\mathcal{P}}
\overline{\pr} \Biggl[ \Biggl|\frac{1}{n}\sum_{j=1}^n
\bigl(w_\pr\bigl(t_{i-1,\pr}^\varepsilon \bigr)
\eins_{[0,t_{i-1,\pr}^\varepsilon]}(U_{j,\pr})
\nonumber
\\[-8pt]
\\[-8pt]
&&\hspace*{55.4pt}{}-\ex_{\overline{\pr}} \bigl[w_\pr\bigl(t_{i-1,\pr}^\varepsilon
\bigr) \eins _{[0,t_{i-1,\pr}^\varepsilon]}(U_{1,\pr}) \bigr] \bigr) \Biggr| \ge \delta/4
\Biggr] = 0,
\nonumber
\\
\label{glivenko cantelli - eq - proof 2 - 2} &&\lim_{n\to\infty} \sup_{\pr\in\mathcal{P}}
\overline{\pr} \Biggl[ \Biggl|\frac{1}{n}\sum_{j=1}^n
\bigl(w_\pr(U_{j,\pr})\eins _{(t_{i-1,\pr}^\varepsilon,t_{i,\pr}^\varepsilon]}(U_{j,\pr})
\nonumber
\\[-8pt]
\\[-8pt]
&&\hspace*{54pt}{}-\ex_{\overline{\pr}} \bigl[w_\pr(U_{1,\pr})
\eins_{(t_{i-1,\pr
}^\varepsilon,t_{i,\pr}^\varepsilon]}(U_{1,\pr}) \bigr] \bigr) \Biggr| \ge\delta/4 \Biggr] = 0.
\nonumber
\end{eqnarray}
We will show only (\ref{glivenko cantelli - eq - proof 2 - 2}).
Assertion (\ref{glivenko cantelli - eq - proof 2 - 1}) can be shown
similarly, noting that the inequality $w_\pr(t_{i-1,\pr}^\varepsilon
)\eins_{[0,t_{i-1,\pr}^\varepsilon]}(U_{1,\pr})\le w_\pr(U_{1,\pr
})\eins_{[0,t_{i-1,\pr}^\varepsilon]}(U_{1,\pr})$ holds since
$w_\pr$ is nonincreasing.

Corollary~\ref{probability bound} ensures (\ref{glivenko cantelli -
eq - proof 2 - 2}) if we can show that the conditions (\ref
{probability bound - cond - 1}) and (\ref{probability bound - cond -
2}) in the corollary hold for $\Pi:=\mathcal{P}$, $(\Omega_\pi
,\mathcal{F}_\pi,\pr_\pi):=(\Omega_\pr,\mathcal{F}_\pr
,\overline{\pr})$ and $\xi_{j,\pi}^{(i)}:=w_\pr(U_{j,\pr})\eins
_{(t_{i-1,\pr}^\varepsilon,t_{i,\pr}^\varepsilon]}(U_{j,\pr})$,
$j\in\N$, for every fixed $i\in\{1,\ldots,m_\varepsilon\}$.
Condition (\ref{probability bound - cond - 1}) follows from
\begin{eqnarray*}
\limsup_{K\to\infty} \sup_{\pr\in\mathcal{P}} \ex_{\overline
{\pr}}
\bigl[\bigl\llvert \xi_{1,\pr}^{(i)}\bigr\rrvert
\eins_{|\xi_{1,\pr}^{(i)}|\ge
K} \bigr] & \le& \limsup_{K\to\infty} \sup
_{\pr\in\mathcal{P}} \ex _{\overline{\pr}} \bigl[\bigl\llvert
w_\pr(U_{1,\pr})\bigr\rrvert \eins_{|w_\pr(U_{1,\pr
})|\ge K} \bigr]
\\
& \le& \limsup_{K\to\infty} \sup_{\pr\in\mathcal{P}} \int
_0^1\bigl |\phi\bigl(F_\pr^\leftarrow(t)
\bigr)\bigr | \eins_{|\phi(F_\pr^\leftarrow
(t))|\ge K} \,\mathrm{d}t
\\
& = & \limsup_{K\to\infty} \sup_{\pr\in\mathcal{P}} \int\phi (y)
\eins_{\phi(y)\ge K} \pr_1(\mathrm{d}y)
\end{eqnarray*}
and assumption (b). Condition (\ref{probability bound - cond - 2}) is
an immediate consequence of assumption (c) and the fact that the
sequence $(U_{j,\pr})$ has the same mixing coefficients under
$\overline{\pr}$ as the sequence $(X_i)$ under $\pr$. This completes
the proof of Theorem~\ref{glivenko cantelli}.


\section{Proof of Theorem \texorpdfstring{\protect\ref{glivenko cantelli - weakl
topology}}{3.9}}\label{proof of glivenko cantelli - weakl topology}

Let $\delta>0$ be arbitrary but fixed. Choose $k_\delta\in\N$ such
that $\sum_{k=k_\delta+1}^\infty2^{-k}<\delta/3$, and notice that
(\ref{glivenko cantelli - weakl topology - eq}) holds if we can show
that the following two conditions hold
%
\begin{eqnarray}
\label{glivenko cantelli - weakl topology - eq - proof 1} \lim_{n\to\infty} \sup_{\pr\in\mathcal{P}} \pr
\biggl[ \biggl\llvert \int f_k \,\mathrm{d}\widehat m_{n}-\int
f_k \,\mathrm{d}\pr_1 \biggr\rrvert \ge\frac{\delta}{3k_\delta
} \biggr] & =
& 0,\qquad k=1,\ldots,k_\delta,
\\
\label{glivenko cantelli - weakl topology - eq - proof 2} \lim_{n\to\infty} \sup_{\pr\in\mathcal{P}} \pr
\biggl[ \biggl\llvert \int \psi \,\mathrm{d}\widehat m_{n}-\int\psi \,\mathrm{d}
\pr_1 \biggr\rrvert \ge\frac{\delta}{3} \biggr] & = & 0.
\end{eqnarray}

To prove (\ref{glivenko cantelli - weakl topology - eq - proof 2}), we
note that the left-hand side in (\ref{glivenko cantelli - weakl
topology - eq - proof 2}) can be written as\vspace*{-1pt}
\[
\lim_{n\to\infty} \sup_{\pr\in\mathcal{P}} \pr \Biggl[ \Biggl
\llvert \frac{1}{n}\sum_{j=i}^n
\psi(X_i)-\ex_{\pr}\bigl[\psi(X_1)\bigr] \Biggr
\rrvert \ge \frac{\delta}{3} \Biggr].
\]
The latter is zero by Corollary~\ref{probability bound} with $\Pi
:=\mathcal{P}$, $(\Omega_\pi,\mathcal{F}_\pi,\pr_\pi):=(\Omega
,\mathcal{F},\pr)$ and the random variables $\xi_{i,\pi}:=\psi
(X_i)$, $i\in\N$. Indeed: assumption (\ref{probability bound - cond
- 1}) of the corollary holds by our assumption (b) and the identity
$\ex_{\pi} [|\xi_{1,\pi}|\eins_{|\xi_{1,\pi}|\ge K}
]=\int\psi\eins_{\psi\ge K} \,\mathrm{d}\pr_1$. Assumption (\ref{probability
bound - cond - 2}) of the corollary holds by assumption (c) and the
fact that, under every $\pr\in\mathcal{P}$, the sequence $(\psi
(X_i))$ has the same mixing coefficients as the sequence $(X_i)$.

Assertion (\ref{glivenko cantelli - weakl topology - eq - proof 1})
can be proven in the same line, noting that $\mathcal{P}$ also
satisfies condition (b) with $\psi$ replaced by $f_k$ for any
$k=1,\ldots,k_\delta$; recall that each $f_k$ has compact support.
This completes the proof of Theorem~\ref{glivenko cantelli - weakl topology}.\vspace*{-1pt}


\begin{appendix}

\section{Uniform weak LLN for strongly mixing random variables}\label
{appendix rio}

\setcounter{equation}{26}
\renewcommand{\theequation}{\arabic{equation}}

Let $(\xi_i)=(\xi_i)_{i\in\N}$ be a sequence of random variables on
some probability space $(\Omega,\mathcal{F},\pr)$. According to
Rosenblatt \cite{Rosenblatt1956}, the sequence $(\xi_i)$ is said to
be \emph{strongly mixing} (or \emph{$\alpha$-mixing}) if the $n$th
mixing coefficient $\alpha_\pr(n):=\sup_{k\ge1}\sup_{A,B} |\pr
[A\cap B]-\pr[A]\pr[B]|$ converges to zero as $n\to\infty$, where
the second supremum ranges over all $A\in\sigma(\xi_1,\ldots,\xi
_k)$ and $B\in\sigma(\xi_{k+n},\xi_{k+n+1},\ldots)$. Recall from
\cite{Bradley2005}, Inequality (1.9), that $\alpha_\pr(n)\le1/4$ for
all $n\in\N$, and use the convention $\alpha_\pr(0):=1/4$. For an
overview on mixing conditions, see \cite{Bradley2005,Doukhan1994}. The
following result is a consequence of Theorem~4 and Lemma~1 in \cite
{Rio1995}; cf. (5.1) on page 936 in \cite{Rio1995}.\vspace*{-1pt}

\begin{theorem}[(Rio)]\label{Rios Markov inequality}
Let $\xi_1,\xi_2,\ldots$ be identically distributed with $\ex[\xi
_1^2]<\infty$, and $\alpha_\pr(n)$ be defined as above.
Let $G$ be the distribution function of $|\xi_1|$, and set $\overline
G:=1-G$. Then, for every $x>0$ and $n\in\N$,\vspace*{-1pt}
\[
\pr \Biggl[ \sup_{k=1,\ldots,n} \Biggl\llvert \sum
_{i=1}^k\bigl(\xi_i-\ex[\xi
_1]\bigr) \Biggr\rrvert \ge2x \Biggr] \le \frac{16}{x^{2}} n \sum
_{j=0}^{n-1}\int_0^{2\alpha_\pr
(j)}
\overline{G}^{ \rightarrow}(t)^2 \,\mathrm{d}t.
\]
\end{theorem}

From Theorem~\ref{Rios Markov inequality}, we can even derive the
following uniform weak LLN for strongly mixing random variables. In the
special case of independent random variables the corollary is already
known from \cite{Chung1951}. In fact, \cite{Chung1951} provides even
a uniform strong LLN.\vspace*{-1pt}

\begin{corollary}\label{probability bound}
Let $\Pi\neq\emptyset$ be an arbitrary index set. Further, for
every $\pi\in\Pi$, let $(\Omega_\pi,\mathcal{F}_\pi,{\pr}_\pi
)$ be a probability space and $\xi_{1,\pi},\xi_{2,\pi},\ldots$ be
a sequence of random variables on $(\Omega_\pi,\mathcal{F}_\pi,\pr
_\pi)$ being identically distributed and strongly mixing with mixing
coefficients $(\alpha_{\pi}(n)):=(\alpha_{\pr_\pi}(n))$. Further
suppose that the following two conditions hold\vspace*{-1pt}
%
\begin{eqnarray}
\label{probability bound - cond - 1} \lim_{K\to\infty} \sup_{\pi\in\Pi}
\ex_\pi \bigl[|\xi_{1,\pi
}|\eins_{|\xi_{1,\pi}|\ge K} \bigr] &=& 0,
\\[-1pt]
\label{probability bound - cond - 2} \lim_{n\to\infty} \sup_{\pi\in\Pi}
\alpha_{\pi}(n) &=& 0. %
\end{eqnarray}
Then, for every $\delta>0$,
%
\begin{equation}
\label{probability bound - eq} \lim_{n\to\infty} \sup_{\pi\in\Pi}
\pr_\pi \Biggl[ \Biggl\llvert \frac
{1}{n}\sum
_{i=1}^n \xi_{i,\pi}-\ex_\pi[
\xi_{1,\pi}]\Biggr\rrvert \ge\delta \Biggr] = 0.
\end{equation}
\end{corollary}

\begin{pf}
We will use a truncation argument. Let $K>0$ be a constant (to be
specified later on) and $\xi_{i,\pi}^K:=\xi_{i,\pi}\eins_{|\xi
_{i,\pi}|\le K}$ be the $K$-truncation of $\xi_{i,\pi}$. Using the
decomposition $\xi_{i,\pi}=\xi_{i,\pi}^K+\xi_{i,\pi}\eins_{|\xi
_{i,\pi}|>K}$ and the triangular inequality, we obtain
\begin{eqnarray*}
\pr_\pi \Biggl[ \Biggl\llvert \frac{1}{n}\sum
_{i=1}^n \xi_{i,\pi}-\ex_\pi [
\xi_{1,\pi}] \Biggr\rrvert \ge\delta \Biggr] & \le& \pr_\pi
\Biggl[ \Biggl\llvert \frac{1}{n}\sum_{i=1}^n
\xi_{i,\pi
}^K-\ex_\pi\bigl[\xi_{1,\pi}^K
\bigr] \Biggr\rrvert \ge\delta/3 \Biggr]
\\
& &{} + \pr_\pi \Biggl[ \Biggl\llvert \frac{1}{n}\sum
_{i=1}^n \xi_{i,\pi
}\eins_{|\xi_{i,\pi}|>K}
\Biggr\rrvert \ge\delta/3 \Biggr]
\\
& &{} + \pr_\pi \bigl[ \ex_\pi\bigl[|\xi_{1,\pi}|
\eins_{|\xi_{1,\pi
}|>K}\bigr]\ge\delta/3 \bigr]
\\
& =: & S_1(\delta,n,K,\pi)+S_2(\delta,n,K,
\pi)+S_3(\delta,K,\pi).
\end{eqnarray*}
By Markov's inequality, $S_2(\delta,n,K,\pi)$ is bounded above by
$3\delta^{-1}\sup_{\pi\in\Pi} \ex_\pi[|\xi_{1,\pi}|\eins
_{|\xi_{1,\pi}|\ge K}]$. So, for given $\varepsilon>0$, one can
choose $K=K_{\varepsilon,\delta}$ in such a way that
\[
\sup_{\pi\in\Pi} S_2(\delta,n,K,\pi) \le\varepsilon/2
\qquad\forall n\in\N,
\]
because we assumed (\ref{probability bound - cond - 1}). By (\ref
{probability bound - cond - 1}), we may and do also assume that
$K=K_{\varepsilon,\delta}$ is chosen such that
\[
\sup_{\pi\in\Pi} S_3(\delta,K,\pi) = 0.
\]
Choosing $x:=n\delta/6$ in Theorem~\ref{Rios Markov inequality}, we
further obtain
\begin{eqnarray*}
\sup_{\pi\in\Pi}S_1(\delta,n,K,\pi) & \le&
\frac{576}{\delta^2} \sup_{\pi\in\Pi} \frac{1}{n}\sum
_{j=0}^{n-1}\int_0^{2\alpha_\pi(j)}
\overline{G}_{K,\pi}^{
\rightarrow}(t)^2 \,\mathrm{d}t
\\
& \le& \frac{1152 K^2}{\delta^2} \sup_{\pi\in\Pi} \frac
{1}{n}\sum
_{j=0}^{n-1}\alpha_\pi(j),
\end{eqnarray*}
where $G_{K,\pi}$ denotes the distribution function of $|\xi_{1,\pi
}^K|$. So, in view of (\ref{probability bound - cond - 2}) and the
Toeplitz lemma, we can find some $n_\varepsilon\in\N$ such that
\[
\sup_{\pi\in\Pi}S_1(\delta,n,K,\pi) \le\varepsilon/2
\qquad \forall n\ge n_\varepsilon.
\]
Thus, (\ref{probability bound - eq}) holds.
\end{pf}


\section{Strong mixing of linear processes}\label{Strong mixing of
linear processes}

Let $(Z_s)_{s\in\Z}$ be i.i.d. random variables on some probability
space $(\Omega,\mathcal{F},\pr)$, and define the linear process
$X_t:=\sum_{s=0}^\infty a_{s}Z_{t-s}$, $t\in\N$, for any real
sequence $(a_s)_{s\in\N_0}$ for which the latter series is $\pr
$-almost surely well defined for every $t\in\N$. Moreover, let
$\alpha_\pr(n)$ be the $n$th strong mixing coefficient of the
sequence $(X_t)$ under $\pr$ as defined at the beginning of
Section~\ref{appendix rio}. The following criterion for $(X_t)$ to be strongly
mixing is an immediate consequence of Lemmas 2.1 and 2.2 in \cite
{PhamTran1985}; notice that we will assume $a_0=1$.

\begin{theorem}\label{strong mixing of linear process}
Let $a_0=1$, and assume that the following assertions hold:
\begin{enumerate}[(b)]
\item[(a)] The distribution of $Z_1$ admits a Lebesgue density $f$ for
which $\int|f(y+h)-f(y)| \,\mathrm{d}y<M|h|$ for all $h\in\R$ and some constant $M>0$.
\item[(b)] $\ex[|Z_1|]<\infty$.
\item[(c)] $\sum_{s=0}^\infty a_s z^s\neq0$ for all $z$ with
$|z|\le1$.
\item[(d)] $\sum_{u=1}^\infty\sum_{s=u}^\infty|a_s|<\infty$.
\end{enumerate}
Then, for every $n\in\N$,
\setcounter{equation}{29}
\begin{equation}
\label{strong mixing of linear process - eq} \alpha_\pr(n) \le \Biggl(2M \ex\bigl[|Z_1|\bigr]
\sum_{s=0}^\infty|b_s| \Biggr)\sum
_{u=n}^\infty\sum_{s=u}^\infty|a_s|,
\end{equation}
where $b_s$ is the coefficient of $z^s$ in the power series expansion
of $z\mapsto1/\sum_{s=0}^\infty a_sz^s$. In particular, $(X_t)$ is
strongly mixing.
\end{theorem}

It can be seen from Lemma~2.1 in \cite{PhamTran1985} that the
right-hand side in (\ref{strong mixing of linear process - eq})
provides an upper bound even for the $\beta$-mixing coefficient, that
is, for the mixing coefficient in the context of absolute regularity.
Notice that (d) implies $\sum_{s=0}^\infty|a_s|<\infty$ which,
together with (c) and Wiener's theorem (cf. \cite{Wiener1933}, page 91
or \cite{Yosida1995}, page 301), ensures that $\sum_{s=0}^\infty
|b_s|<\infty$. Further conditions for a linear process to be strongly
mixing can be found in \cite{Gorodetskii1977,PhamTran1985,Withers1981}.

\begin{remarknorm}\label{strong mixing of linear process - remark}
(i) Condition (d) is satisfied if $|a_s|\le C s^{-\gamma}$ for all
$s\in\N_0$ and some constants $C>0$, $\gamma>2$. In this case, $\sum_{u=n}^\infty\sum_{s=u}^\infty|a_s|\le C_\gamma n^{2-\gamma}$ holds
for all $n\in\N$.

(ii) Condition (d) is also satisfied if $|a_s|\le C q^{s}$ for all
$s\in\N_0$ and some constants $C>0$, $q\in(0,1)$. In this case,
$\sum_{u=n}^\infty\sum_{s=u}^\infty|a_s|\le C(1-q)^{-2}q^n$ holds
for all $n\in\N$.
\end{remarknorm}

\begin{examplenorm}\label{strong mixing of linear process - example}
If $a_s=a q^{s}$, $s\ge1$, with $a\neq0$ and $|q|<1$, then we have
for $|z|\le1$
\[
\sum_{s=0}^\infty a_sz^s
= (a_0-a)+\frac{a}{1-qz} = \frac{a_0-\{
(a_0-a)q\}z}{1-qz}.
\]
From here one easily derives that $1/\sum_{s=0}^\infty a_sz^s$ admits
the representation $\sum_{s=0}^\infty b_sz^s$ with $b_0=1/a_0$ and
$b_s=-a(a_0-a)^{s-1}a_0^{-(s+1)}q^s$, $s\ge1$, using the convention $0^0:=1$.

If specifically $a_0=1$, $a=(\phi_1+\theta_1)/\phi_1$ and $q=\phi
_1$ for real numbers $\phi_1$ and $\theta_1$ satisfying $0<|\phi
_1|<1$ and $|\theta_1|<1$, then $\sum_{s=0}^\infty|b_s|=1+|\phi
_1+\theta_1|/(1-|\theta_1|)$.
\end{examplenorm}

\begin{examplenorm}[(ARMA process)]\label{ARMA process - appendix}
To illustrate conditions (c)--(d) in Theorem~\ref{strong mixing of
linear process},
let us consider an $\operatorname{ARMA}(p,q)$ process $X_t=\sum_{s=1}^p\phi_s
X_{t-s} + \sum_{s=0}^q \theta_s Z_{t-s}$
with $\theta_0=1$ and square-integrable and centered i.i.d.
innovations $(Z_s)_{s\in\Z}$. Define the characteristic polynomials
$\phi(z):=1-\sum_{s=1}^p\phi_s z^s$ and $\theta(z):=\sum_{s=0}^q
\theta_s z^s$, and assume that $\phi(\cdot)$ and $\theta(\cdot)$
have no common zeros. Further assume that $(X_t)$ is both causal and
invertible. By the causality we have that $\phi(z)\neq0$ for all
complex $z$ with $|z|\le1$, and that $X$ admits the $\operatorname{MA}(\infty)$
representation $X_t=\sum_{s=0}^\infty a_{s}Z_{t-s}$ with the
coefficients $a_s$ determined by
%
\begin{equation}
\label{eq for infinute polynomial a} \frac{\theta(z)}{\phi(z)} = a(z) := \sum_{s=0}^\infty
a_s z^s;
\end{equation}
see \cite{BrockwellDavis2006}, Theorem~3.1.1. By the invertibility we
have in addition that $\theta(z)\neq 0$, and hence $a(z)\neq0$,
for all complex $z$ with $|z|\le1$; see \cite{BrockwellDavis2006}, Theorem~3.1.2. That is, under the imposed assumptions the
ARMA process can be regarded as a linear process satisfying condition
(c) in Theorem~\ref{strong mixing of linear process}.

If specifically $p=q=1$, $|\phi_1|,|\theta_1|<1$, and $\phi_1\neq\theta_1$, then the coefficients $a_s$ determined by (\ref{eq for
infinute polynomial a}) read as $a_0=1$ and $a_s=(\phi_1+\theta
_1)\phi_1^{s-1}$ ($=((\phi_1+\theta_1)/\phi_1)\phi_1^{s}$ if $\phi
_1\neq0$), $s\ge1$. In this case, the $\operatorname{ARMA}(1,1)$ process can be
seen as a linear process which satisfies not only condition (c) but
also condition (d) of Theorem~\ref{strong mixing of linear process}.
By part (ii) of Remark~\ref{strong mixing of linear process - remark},
we have in particular $\sum_{u=n}^\infty\sum_{s=u}^\infty
|a_s|=(|\phi_1+\theta_1|/(1-|\phi_1|)^{2}) |\phi_1|^{n-1}$ for all\vspace*{2pt}
$n\in\N$. Moreover, Example~\ref{strong mixing of linear process -
example} yields $\sum_{s=0}^\infty|b_s|=1+|\phi_1+\theta
_1|/(1-|\theta_1|)$, where $b_s$ is the coefficient of $z^s$ in the
power series expansion of $z\mapsto1/\sum_{s=0}^\infty a_sz^s$.
\end{examplenorm}
\end{appendix}

\section*{Acknowledgements}

The author thanks the editor and the reviewers for their constructive
suggestions and helpful comments.


%

\printhistory

\end{document}